\documentclass[12pt]{article}

\usepackage{lipsum}
\usepackage{fancyhdr}
\usepackage{booktabs}
\usepackage{array}
\usepackage{amssymb}
\usepackage{parskip}
\usepackage{hyperref}
\usepackage{enumitem}
\usepackage{mathtools} 
\usepackage{xparse,textcomp}
\DeclarePairedDelimiter\floor{\lfloor}{\rfloor}

\usepackage[lofdepth,lotdepth]{subfig}

\usepackage{cite}

\usepackage{parskip}
\usepackage[margin=0.8in]{geometry}
\newcommand{\forceindent}{\leavevmode{\parindent=1.5em\indent}}
\usepackage{algorithm} 
\usepackage{algpseudocode}
\newcommand{\R}{\mathbb R}

\newcommand{\maxmath}{\mathop{}\!\mathrm{max}}
\newcommand{\minmath}{\mathop{}\!\mathrm{min}}

\usepackage{amsmath,amsthm}
\makeatletter

\usepackage[belowskip=-30pt,aboveskip=10pt]{caption} 
\usepackage{amsfonts}
\usepackage{amsopn}

\usepackage{amssymb,amsmath}

\renewcommand\section{\leftskip 0pt\@startsection {section}{1}{\z@}%
	{-3.5ex \@plus -1ex \@minus -.2ex}%
	{2.3ex \@plus.2ex}%
	{\normalfont\Large\bfseries}}

\renewcommand\subsection{\leftskip 0pt\@startsection{subsection}{2}{\z@}%
	{-3.25ex\@plus -1ex \@minus -.2ex}%
	{1.5ex \@plus.2ex}%
	{\normalfont\large\bfseries}}

\renewcommand\subsubsection{\leftskip 0pt\@startsection{subsubsection}{3}{\z@}%
	{-3.25ex\@plus -1ex \@minus -.2ex}%
	{1.5ex \@plus .2ex}%
	{\normalfont\large\bfseries}}
\usepackage{graphicx}

\usepackage{geometry}
\newtheorem{definition}{Definition}

\newtheorem{Theorem}{Theorem}[section]

\newtheorem{corollary}[Theorem]{Corollary}

\theoremstyle{remark}
\newtheorem{remark}[Theorem]{Remark}

\begin{document}
	
\title{ \centering{Manifold Repairing, Reconstruction and Denoising from Scattered Data in High-Dimension}}
\author{Shira Faigenbaum-Golovin${^{1, *}}$~~David Levin${^1}$ \\
	\small{${^1}$ School of Mathematical Sciences, Tel Aviv University, Israel}
	\\
	\small{${^*}$ Corresponding author, E-mail address: alecsan1@post.tau.ac.il} 
}


\maketitle
\begin{abstract}
We consider a problem of great practical interest: the repairing and recovery of a low-dimensional manifold embedded in high-dimensional space from noisy scattered data. Suppose that we observe a point cloud sampled from the low-dimensional manifold, with noise, and let us assume that there are holes in the data. Can we recover missing information inside the holes? While in low-dimension the problem was extensively studied, manifold repairing in high dimension is still an open problem. We introduce a new approach, called Repairing Manifold Locally Optimal Projection (R-MLOP), that expands the MLOP method \cite{faigenbaumgolovin2020manifold}, to cope with manifold repairing in low and high-dimensional cases. The proposed method can deal with multiple holes in a manifold. We prove the validity of the proposed method, and demonstrate the effectiveness of our approach by considering different manifold topologies, for single and multiple holes repairing, in low and high dimensions.

\end{abstract}

\noindent\textbf{keywords:} Manifold learning, Manifold denoising, Manifold reconstruction, Manifold repairing, High dimension, Dimensional reduction

(Numerical analysis - Numerical approximation and computational geometry)

\section{Introduction}
\label{repairing_intro}

Imagine that we observe partial data sampled, with noise, from a manifold in high dimensional space. Is it possible to reconstruct this manifold exactly or at least accurately? For example, in the case of acquiring oil and gas well data, where each drilling operation is expensive, high-resolution sampling can be extremely pricey. In general, everybody would agree that recovering data from an incomplete sample is impossible. In this paper, we show that if the unknown data is known to lie on a manifold in high dimension, then it is possible to reconstruct it and approximate the samples that we have not seen.

Before we turn to discuss the problem for high-dimensional data, we first consider the simpler, yet challenging questions of surface reconstruction and recovering missing information. The challenge of plain surface reconstruction gained a lot of attention in the recent years, with many methods proposed to address it (see, the papers of \cite{alexa2003computing, berger2017survey, cohen1999progressive, levin2004mesh, lipman2007parameterization}). There are many challenges when dealing with this problem, among them the need for preservation of sharp features, reconstruction from uniform/non-uniform data sampling, undersampling, fast-changing curvature \cite{faigenbaumAnisoMLS}, adjacent surface segments, and the presence of noise. For example, usually reconstruction methods smooth out the existing corners and sharp features present on the surface. Several papers raised the sharp features preservation challenge and suggested methods to address it (e.g., \cite{huang2013edge, yadav2018constraint}). 

\forceindent Missing information in the data, i.e., the presence of holes, is an additional challenge in surface reconstruction. During surface acquisition, sampling the entire geometry is not always possible due to physical obstacles, occlusions, scanning angle issues, scanning costs, or to simply the fact that the available data is incomplete. The problem of hole filling can be formulated as follows: Given data sampled from a surface with holes, the task is to generate new points so as to reconstruct the data inside the hole. The problem of filling holes and recovering lost information in low dimensions has been dealt with by many researchers (for a concise survey, see \cite{guo2018survey}). The problem itself is usually treated as a two-step challenge, namely, locating the hole, and then recovering the missing data. The solution for both steps can be roughly classofoed according to the way in which the input data are given: as a point-cloud or as a mesh. 

\forceindent The hole identification problem is further challenged by the real-life scenarios since usually information regarding surface geometry is unavailable. Subsequently, recognizing the areas that need to be amended and the ones that should be omitted (since this is the true geometry of the surface), becomes an ill-defined and non-trivial task. In case some information regarding the connectivity of the data is available, mesh representation comes in handy. Methods like those presented in \cite{zhou2014algorithm, he2011virtual, jun2005piecewise} utilize this information for the hole identification by recognizing hollow edges, extract boundary via $k$-nearest neighbors search, or labeling boundaries based on the triangulation structure. On the other hand, if the data are given as a point cloud, other methods are applicable. In \cite{chalmoviansky2003filling} the authors suggested locating boundaries via sparse area heuristics, in \cite{moenning2004intrinsic} an enhanced $k$-nearest neighborhood is used, while in \cite{bendels2006detecting} a set of criteria are derived for automatic hole detection (the criteria rely on symmetric $k$-$\epsilon$ neighborhoods, their angle criterion and the deviation of a point and its neighborhood average, as well as the shape criterion, based on eigenvalues).

\forceindent Once a hole is located, reconstructing the information inside the hole can be achieved by various methods. In the case where the surface is represented as a mesh, the paper \cite{zhou2014algorithm} uses a triangle growth procedure based on the boundary edge angle, in \cite{he2011virtual} the hole is filled under constraint on the boundary, in \cite{jun2005piecewise, tang2017repair} holes are amended in a piecewise manner, i.e., by splitting holes into smaller parts and fixing latter, while in \cite{attene2013polygon} a mesh repairing technique is proposed. When the surface is given as a point cloud, the authors of \cite{chalmoviansky2003filling} suggested fitting an algebraic surface patch to the neighborhood of the boundary; this avoids parametrization and allows one to deal with holes with complicated topology. Other solutions rely on the minimal surface solutions (Plateau’s problem \cite{thi1991minimal, harrison2014soap}), diffusion \cite{davis2002filling}, or an MLS-based point-cloud resampling technique \cite{wang2007filling, moenning2004intrinsic}. In addition, some methods address the surface repairing problem as the problem of reconstructing the surface from boundary conditions \cite{wang2014restructuring, ostrov1999boundary}. 

\forceindent The problem of completing missing information arises not only in geometrical settings, but also in other fields, for example, in matrix computations \cite{candes2010matrix}, or in image processing. In the latter case, the problem is known as image inpainting, and various approaches were proposed it \cite{bertalmio2000image, richard2001fast, criminisi2004region, xu2010image}. Of special interest is the elegant solution suggested in \cite{bertalmio2000image}, where the hole edge information is propagated inside the hole using the gradient of the Laplacian in the direction of the edge. Thanks to the geometrical nature of the solution, it can be adopted to the surface case.

\forceindent The existing solutions for surface reconstruction have a hard time coping with challenges raised by real-life applications, especially in high dimension. First, they rely on some prior knowledge regarding the surface (e.g., a triangulation). In addition, these solutions do not deal with noise or outliers, which are usually present in the data. Moreover, since they were tailored for the surface repairing challenge, most of the known methods cannot be extended to the high-dimensional case (due to practical reasons connected with inefficiency of meshes in high dimensions and to the curse of dimensionality). Thus, manifold repairing in high dimensions is still an open problem. 

\forceindent In this paper, we will address the problem of hole repairing in low and high dimensions. We introduce a new approach, which extends the Manifold Locally Optimal Projection (MLOP) \cite{faigenbaumgolovin2020manifold}, to cope with manifold repairing in low and high-dimensional cases. The MLOP method is a multidimensional extension of the Locally Optimal Projection algorithm \cite{lipman2007parameterization}. The MLOP bypasses the curse of dimensionality and avoids the need for dimension reduction. It is based on a non-convex optimization problem, which leverages a generalization of the outlier-robust $L_1$-median to higher dimensions while generating noise-free quasi-uniformly distributed points reconstructing the unknown low-dimensional manifold. 

\forceindent The vanilla MLOP does not repair the manifold, since by its design it maintains the proximity to the original points. For example, see the MLOP method reconstruction of the Stanford bunny \cite{levoy2005stanford}, in Figure \ref{fig:repair_bunny} (A-B), where one can see that the MLOP method maintained the scattered points geometry, and the existing hole was not amended. To overcome this, we suggest enhancing the MLOP method to address data repairing problem by adding another force of attraction which will push the boundary points towards their convex hull (see Figure \ref{fig:repair_bunny} (C)). It should be noted that since the hole identification problem is ill-posed, and the best way to fix the missing information is by the user manually identifying the holes that need to be amended. In what follows we first assume that the location of the hole is known, and later propose a method for hole identification.
\vspace{-12mm}
\begin{figure}[H]
	\centering
	\captionsetup[subfloat]{farskip=0pt,captionskip=0pt, aboveskip=0pt}
	\subfloat[][]{ \includegraphics[width=0.33\textwidth]{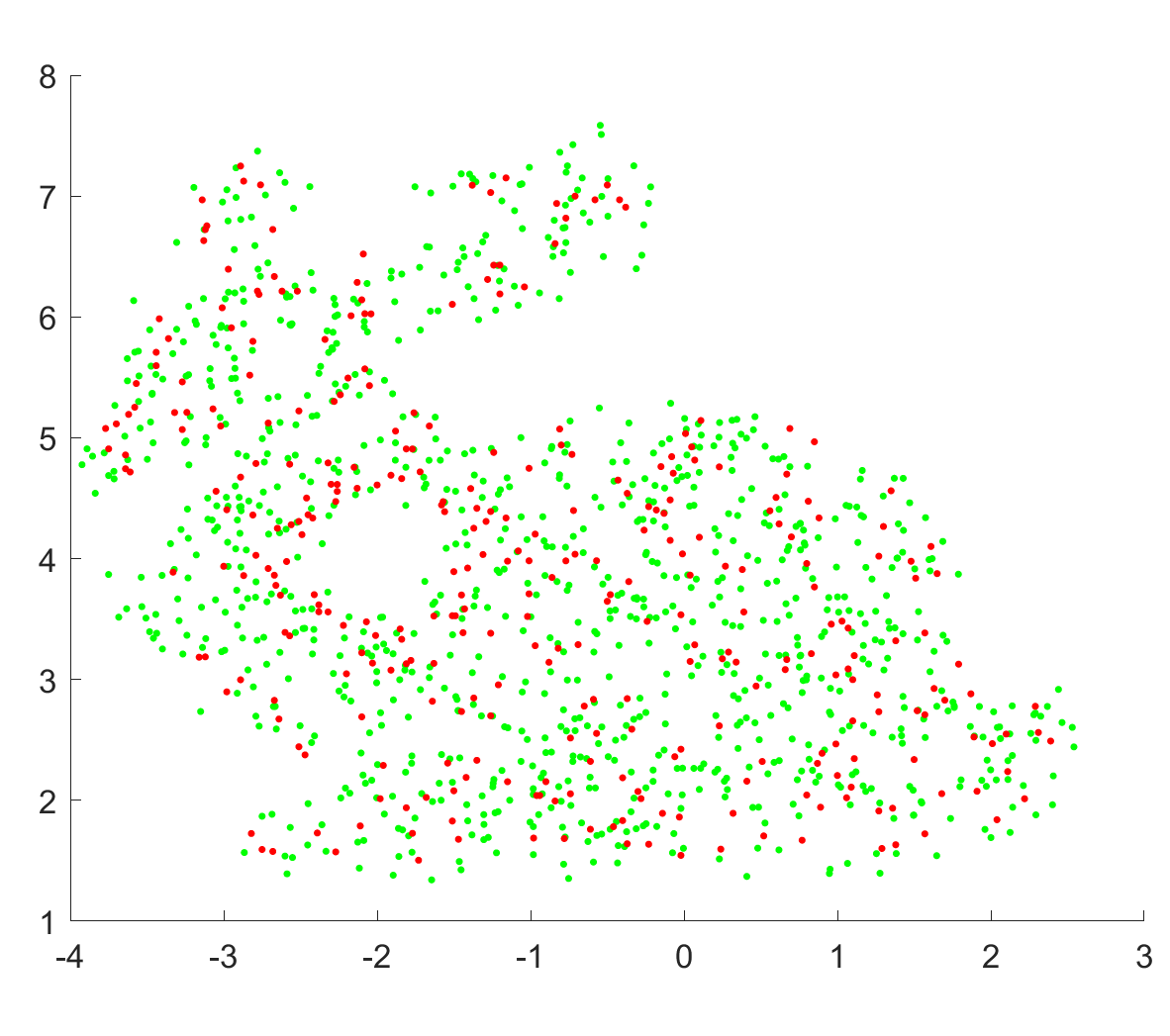} } \hspace{-1em}
	\subfloat[][]{ \includegraphics[width=0.33\textwidth]{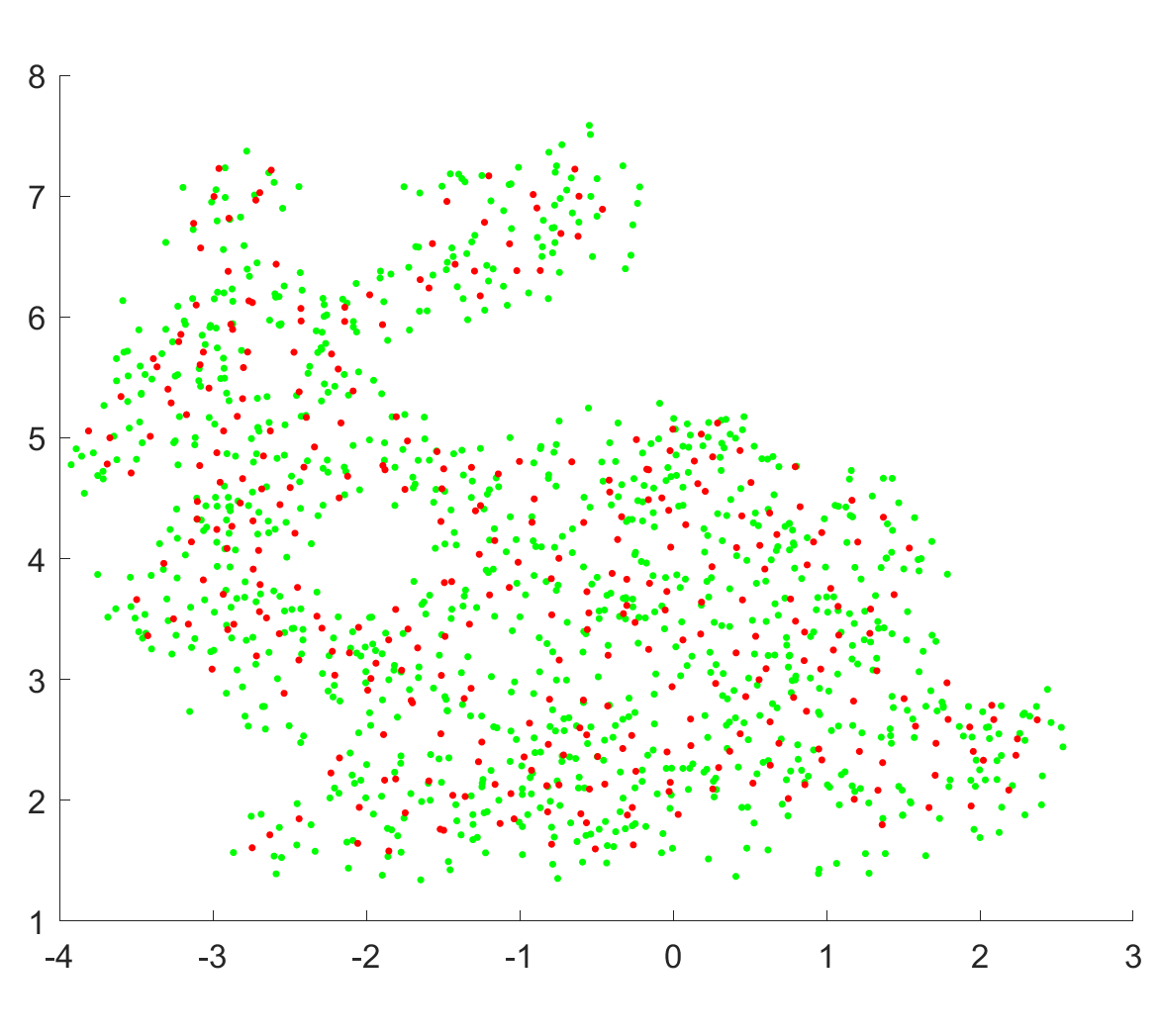} } \hspace{-0.9em}
	\subfloat[][]{\includegraphics[width=0.33\textwidth]{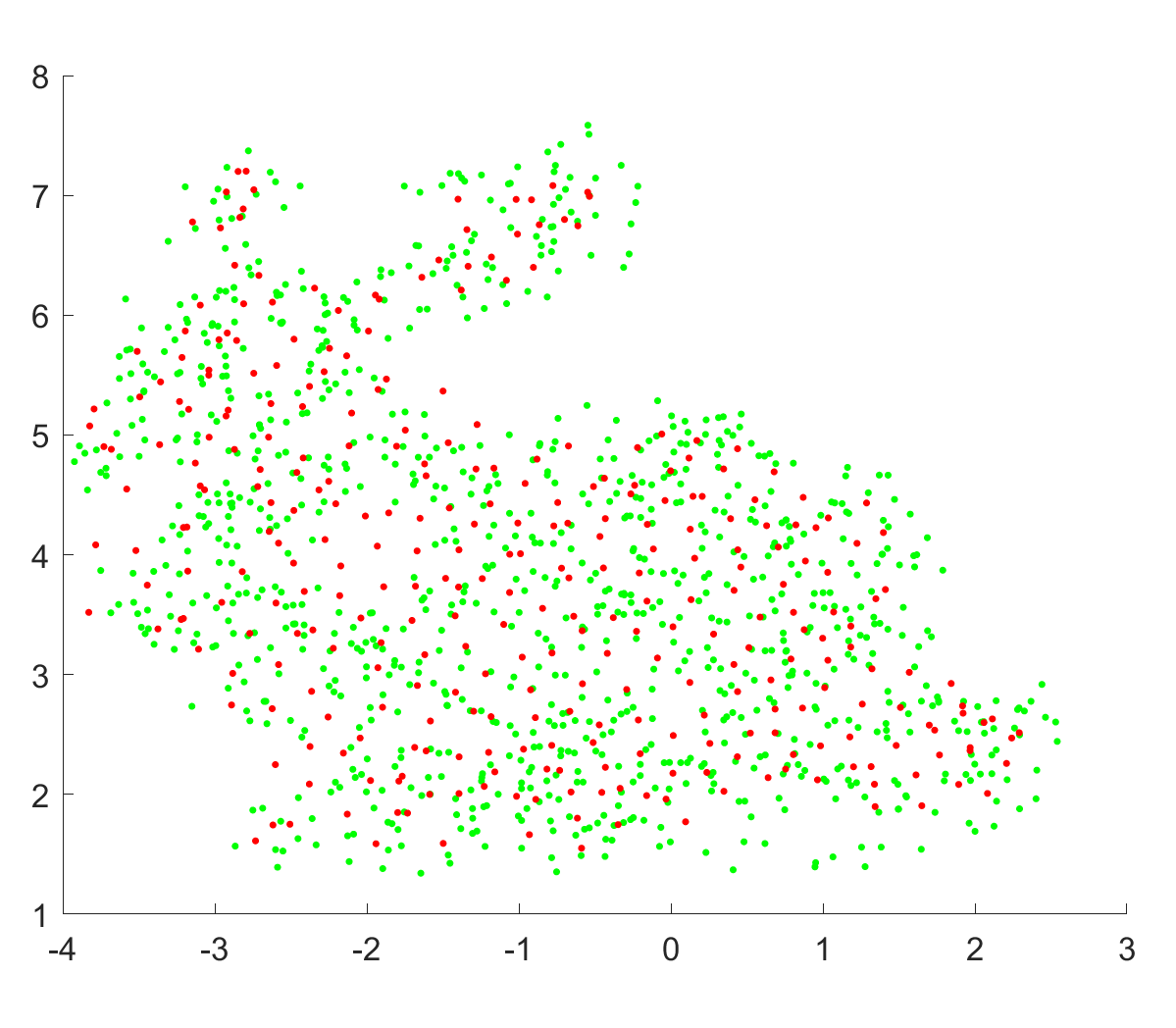}}
	\caption{Amending the scattered data for the Stanford bunny. (A) The initial scattered data of size  1K sampled from the bunny model (the initial $P$-points in green, the initial $Q$-points in red). (B) The picture produced by the plain MLOP, which results in a quasi-uniform sampling of the bunny: the hole is not amended. (C) The picture produced by the R-MLOP algorithm.}
	\label{fig:repair_bunny}
\end{figure}

\section{Preliminaries}
\subsection{Preliminaries --- the MLOP framework}
\label{MLOP_pre}	
The Locally Optimal Projection (LOP) method  was introduced in \cite{lipman2007parameterization} to approximate two-dimensional surfaces in $\mathbb{R}^3$ from point-set data. The procedure does not require the estimation of local normals  and planes,  or parametric representations. Its main advantage  is that it performs well in the case of noisy samples. In \cite{faigenbaumgolovin2020manifold} the LOP mechanism was generalized to devise the so-called \textit{Manifold Locally Optimal Projection} (MLOP) method. This method was further enhanced to deal with approximation of functions in presence of noise in \cite{faigenbaum2020approximation}. Here, we give a concise overview of the MLOP method and its key properties.

\forceindent First, we introduce the $h$-$\rho$ condition, defined for scattered-data approximation of functions (which is an adaptation of the condition in \cite{levin1998approximation} for low-dimensional data), to handle finite discrete data on manifolds.

\begin{definition}\label{def:def4}
	\textbf{$h$-$\rho$ sets of fill distance $h$  and density $\leq \rho$} with respect to a manifold $\mathcal{M}$. Let $\mathcal{M}$ be a manifold in $\mathbb{R}^n$ and consider sets of data points sampled from $\mathcal{M}$. We say that such a set $P=\{{P_j }\}_{i=1}^J$ is an 
	$h$-$\rho$ set if:\\
	1.   $h$ is the fill distance, i.e., $h=\text{median}_{p_j \in P} \minmath_{p_j \in P \backslash\{p_i\}} \|p_i-p_j\|$. \\
	2.   $\#\{P \cap \bar{B}(y,kh)\}\leq \rho k^n, \quad k \geq 1, \quad y\in \mathbb{R}^n$. \\
	Here $\#Y$ denotes the number of elements in a set $Y$ and $\bar{B}(x,r)$ denotes the closed ball of radius $r$ centered at $x$. 
\end{definition}
\forceindent Note that the last condition regarding the point  separation $\delta$ defined in \cite{levin1998approximation}, which states that there exists $\delta >0$ such that  $\|p_i-p_j \| \geq \delta, \quad 1\leq i \leq j \leq J$, is redundant in the case of finite data. We also note   that the vanilla definition of the  fill distance uses the supremum $\sup$ in its expression;  here we use the median in order to deal with the presence of outliers.

\forceindent The setting of the high-dimensional reconstruction problem is the following: Let $\mathcal{M}$ be a manifold  in $\mathbb{R}^n$  of unknown intrinsic dimension $d \ll n$. There  is given a noisy point-cloud $P=\{p_j\}_{j=1}^J \subset \mathbb{R}^n$ situated near  the manifold $\mathcal{M}$   such that   $P$ is a $h$-$\rho$ set. We wish to find a new point-set $Q=\{q_i\}_{i=1}^I \subset \mathbb{R}^n$  which will serve as a noise-free approximation of $\mathcal{M}$. We seek a solution in the form of a new, quasi-uniformly distributed point-set $Q$ that will replace the given data $P$  and provide a noise-free approximation of $\mathcal{M}$. This is achieved by leveraging the well-studied weighted $L_1$-median \cite{vardi2000multivariate} used in the LOP algorithm and requiring a quasi-uniform distribution of points $q_i \in Q$. 
These ideas are encoded by the cost function
\begin{equation} 
\label{eq:1}	 G(Q) = E_1(P,Q)+\Lambda E_2(Q) = \sum\limits_{q_i \in Q} \sum\limits_{p_j \in P} \|q_i-p_j\|_{H_{\epsilon}} w_{i,j} +
\sum\limits_{q_i \in Q} \lambda_i \sum\limits_{q_{i'} \in Q \backslash \{q_i\}} \eta(\|q_i-q_i'\|) \widehat w_{i,i'}\,, 
\end{equation}
where the weights $w_{i,j}$ are rapidly decreasing smooth functions. The MLOP implementation uses  $w_{i,j} = \exp\big\{\!-\|q_i-p_j\|^2/{h_1^2}\big\}$
and $\widehat w_{i,i'} = \exp\big\{\!-\|q_i-q_i'\|^2/{h_2^2}\big\}$.
The  $L_1$-norm used in \cite{lipman2007parameterization}  is replaced by the ``norm" $\| \cdot \|_{H_{\epsilon}}$ introduced in \cite{levin2015between} as $\|v\|_{H_\epsilon}=\sqrt{v^2+\epsilon}$, where $\epsilon >0$ is a fixed parameter (in our case we take $\epsilon=0.1$). As shown in \cite{levin2015between}, using $\| \cdot \|_{H_{\epsilon}}$ instead of $\| \cdot \|_1$ has the advantage  that one works  with a smooth cost function  and outliers can be removed. In addition, $h_1$ and $h_2$ are the support size parameters of $w_{i,j}$ and $\widehat w_{i,i'}$ which guarantee a sufficient amount of $P$ or $Q$ points for the reconstruction (for more details, see the subsection  ``Optimal Neighborhood Selection"  in \cite{faigenbaumgolovin2020manifold}). Also, $\eta(r)$ is a decreasing function such that $\eta(0)= \infty $; in our case we take $\eta(r) = \frac {1} {3r^3}$. Finally, $\{\lambda_i\}_{i=1}^I$ are constant balancing parameters.

In order to solve the problem with the cost function \eqref{eq:1}, we look for a point-set $Q$ that minimizes $G(Q)$. 
The solution $Q$ is found via the gradient descent iterations
\begin{equation}
q_{i'}^{(k+1)}=q_{i'}^{(k)}-\gamma_k \nabla G(q_{i'}^{(k)}),\qquad i'=1,\dots,I \,, 
\end{equation}
where the initial guess $\{q_i^{(0)}\}_{i-1}^I=Q^{(0)}$ consists of points   sampled from $P$. \\
The gradient of $G$ is  
\begin{equation}  \label{eq:GradG} \nabla G(q_{i'}^{(k)}) = \sum\limits_{j=1}^J {\big(q_{i'}^{(k)}-p_j\big) \alpha_j^{i'}}- \lambda_{i'}\sum\limits_{\substack {i=1 \\ i\neq i'}}^I{ \big(q_{i'}^{(k)} - q_{i}^{(k)} \big) \beta_i^{i'}}\,, \end{equation}
with the coefficients $\alpha_j^{i'}$ and $\beta_j^{i'}$  given by the formulas
\begin{equation}
\label{eq:alpha}
\alpha_j^{i'}  = \frac{w_{i,j}}{ \|q_i-p_j\|_{H_{\epsilon}}} \left(1-\frac{2}{h_1^2}\|q_i-p_j\|_{H_{\epsilon}}^2\right)
\end{equation}
and 
\begin{equation}
\label{eq:beta}
\beta_i^{i'} = \frac{\widehat w_{i,i'}}{ \|q_i-q_{i'}\|} \left( \left| {\frac{\partial \eta \left( \|q_i-q_{i'}\| \right) } {\partial r}}\right|  +
\frac{2\eta \left( \|q_i-q_{i'}\| \right) } {h_2^2}   \|q_i-q_{i'}\|
\right),	 
\end{equation}
for 	$i=1,...,I$, $i\ne i'$.
In order to balance the two terms in $\nabla G(q_{i'}^{(k)})$, the factors $\lambda_{i'}$ are initialized in the first iteration as
\begin{equation}  \label{eq:2}
\lambda_{i'} = -\,\frac{\bigg\|\sum\limits_{j=1}^J {\big(q_{i'}^{(k)}-p_j\big) \alpha_j^{i'}} \bigg\|}{\bigg\|\sum\limits_{i=1}^I{\big(q_{i'}^{(k)} - q_{i}^{(k)} \big) \beta_i^{i'}} \bigg\|}\,.
\end{equation}
Balancing the contribution of the two terms is important in order to maintain equal influence of the attraction and repulsion forces in $G(Q)$.
The step size in the direction of the gradient $\gamma_k$ is calculated as indicated in \cite{barzilai1988two}:
\begin{eqnarray} \label{gamma_k} \gamma_k = \frac{\langle \bigtriangleup q_{i'}^{(k)},  \bigtriangleup G_{i'}^{(k)} \rangle} {\langle \bigtriangleup G_{i'}^{(k)},  \bigtriangleup G_{i'}^{(k)} \rangle}\,,  \end{eqnarray}
where  $\bigtriangleup q_{i'}^{(k)}  = q_{i'}^{(k)}  - q_{i'}^{(k-1)}$ and $\bigtriangleup G_{i'}^{(k)}  = \nabla G_{i'}^{(k)}  - \nabla G_{i'}^{(k-1)}$.\\
\vspace{-6mm}
\begin{remark}\label{def:red_dim}
	The reasoning in terms of Euclidean distances, which is the cornerstone of the MLOP method, works well in low dimensions, e.g., for the reconstruction of surfaces in 3D, but breaks down in high dimensions once  noise  is present. To deal with this issue, a dimension reduction is performed via random linear sketching \cite{woodruff2014sketching}. It should be emphasized that the dimension reduction procedure is utilized solely for the calculation of norms, and the manifold reconstruction is performed in the high-dimensional space. Thus, given a point $x \in \mathbb{R}^n$, we project it to a lower dimension $m \ll n$ using a random matrix, $S$, with certain properties. Subsequently, the norm of  $\|S^t x\|$ will approximate $\|x\|$. The construction of $S$ is carried out in the following steps:
	\begin{enumerate}[noitemsep]
		\item Sample $G\in \R^{J\times m}$ with $G\sim N(0, 1)$.
		\item Compute $B \in \R^{n \times m}$ as $B:=P^{\rm t}G$.
		\item Calculate the QR decomposition of $B$ as $B = SR$, and use $S$ as the dimension reduction matrix.
	\end{enumerate}
\end{remark}

\textbf{Preliminaries --- Optimal Neighborhood Selection}

\forceindent The support sizes $h_1$  and $h_2$ of  the functions $w_{i,j} = \exp\big\{\!-\|q_i-p_j\|^2/{h_1^2}\big\}$
and $\widehat w_{i,i'} = \exp\big\{\!-\|q_i-q_i'\|^2/{h_2^2}\big\}$, respectively, are closely related to the fill distance of the $P$-points and the $Q$- points. 
Due to the importance of optimal selection of these parameters, we quote here several definitions and results from \cite{faigenbaumgolovin2020manifold}. Unlike the standard definition of the notion of fill distance in scattered-data function approximation \cite{levin1998approximation}, we introduce
\begin{definition}\label{def:def1}
	The fill distance of the set $P$ is 	
	\begin{eqnarray} \label{h_0} h_0=\text{median}_{p_i\in P} \minmath_{p_j\in P \backslash \{p_i\}} \|p_i-p_j \| \,. \end{eqnarray}
	Note  that the vanilla definition of fill distance uses the supremum in the definition (instead of the \textup{median}). However, as mentioned above, in our case we replace  the supremum with  the median so as to deal with the presence of outliers.
\end{definition}

\begin{definition}\label{def:def2}
	Given two point-clouds, $P=\{p_j\}_{j=1}^J \subset \mathbb{R}^n$ and $Q=\{q_i\}_{i=1}^I \subset \mathbb{R}^n$, situated near a manifold $\mathcal{M}$ in $\mathbb{R}^n$, such that their sizes obey the constraint $I\leq J$, denote $\nu= \floor*{\frac{J} {I}}$. Then we say that the radius that guarantees approximately $\nu$ points from $P$ in the support of each point $q_i$ is $\widehat{h}_0=c_1 h_0$, with $c_1$ given by
	\begin{eqnarray} \label{c_1}  \widehat{h}_0=c_1 h_0,~{\rm with}~ ~c_1=\text{argmin} \{c: \#(\bar{B}_{ch_0}(q_i) \cap P)\geq \nu,\,  \forall q_i\in Q\}\,. \end{eqnarray}
\end{definition}

\begin{remark}\label{def:def3}
	Let $\sigma$ be the variance of the Gaussian $w(r) = \exp\{- {r^2}/{\sigma^2}\}$. For the normal distribution, four standard deviations away from the mean account for $99.99\%$ of the set.
	In our case, by the definition of $w_{i,k}$, since $h$ is the square root of the variance, $4\sigma= 4\frac {h}{\sqrt{2}}= 2 \sqrt{2}h_1$ covers $99.99\%$ of the support size of $w_{i,k}$.
\end{remark}

The following theorem, proved in \cite{faigenbaumgolovin2020manifold}, indicates how the parameters $h_1$ and $h_2$  should be selected.

\begin{Theorem}\label{lma0}
	Let $\mathcal{M}$ be a $d$-dimensional manifold in $\mathbb{R}^n$. Suppose given two point-clouds, $P=\{p_j\}_{j=1}^J \subset \mathbb{R}^n$ and $Q=\{q_i\}_{i=1}^I \subset \mathbb{R}^n$, situated near   $\mathcal{M}$, such that their sizes obey the constraint $I\leq J$, and let $\nu= \floor*{\frac{J} {I}}$. Let $w_{i,j}$ be the locally supported weight function given by $w_{i,j} = \exp\{- {\|q_i-p_j\|^2}/{h^2}\}$. Then a neighborhood size of $h = 2 \sqrt{2}\widehat{h}_0$ guarantees $ 2^{1.5d}\nu$ points in the support of $w_{i,j}$, where $\widehat{h}_0=c_1h_0$, with $c_1$ given by  \textup{\eqref{c_1}}.
\end{Theorem}

\textbf{Theoretical Analysis of the MLOP Method}

For the sake of completeness, we mention here several important results regarding the convergence of the MLOP method, its order of approximation,  rate of convergence  and complexity  (for more details, see \cite{faigenbaumgolovin2020manifold}). 

\begin{Theorem}[Convergence to a stationary point]\label{thm:conv}
	Let $\mathcal{M}$ be a  in $\mathbb{R}^n$ of unknown intrinsic dimension $d$. Suppose that the scattered data points $P = \{{p_j }\}_{j =1}^J$ were sampled near the manifold $\mathcal{M}$, $h_1$ and $h_2$ are set as  in Theorem  \textup{\ref{lma0}}, and the $h$-$\rho$ set condition is satisfied with respect to $\mathcal{M}$. Let the points $Q^{(0)}=\{q_i^{(0)} \}_{i=1}^I$ be sampled from $P$. Then the gradient descent iterations \textup{\eqref{eq:2}} converge almost surely to a local minimizer $Q^*$.
\end{Theorem}

\begin{Theorem}[Order of approximation]\label{thm:approx_order}
	Let $P=\{p_j\}_{j=1}^J$ be a set of points that are sampled \textup{(}without noise\textup{)} from a $d$–dimensional $C^{2}$ manifold $\mathcal{M}$, and satisfy the $h$-$\rho$ condition. Then for a fixed $\rho$  and a finite support of size $h$ of the weight functions  $w_{i,j}$, the set $Q$ defined by the \textup{MLOP} algorithm has an order of approximation $O(h^{2})$ to $\mathcal{M}$.
\end{Theorem}

\begin{definition}\label{def:def6}
	A differentiable function $f(\cdot)$ is called $L$-smooth if for any $x_1, x_2$
	\begin{eqnarray*}  \|\nabla	 f(x_1) - \nabla f(x_2)\| \leq L\|x_1 - x_2\| \,.\end{eqnarray*}
\end{definition}

\begin{Theorem}[Rate of convergence]\label{thm:convRate}
	Suppose the point-set $P=\{{p_j }\}_{j =1}^J$ is  sampled near a $d$-dimensional manifold in $\mathbb{R}^n$ and the assumptions of  Theorem~{\rm \ref{thm:conv}} are satisfied. Suppose the cost function $G$ defined   in {\rm (\ref{eq:1})} is an $L$-smooth function. For  $\epsilon>0$, let $Q^*$ be a local fixed-point solution of the gradient descent iterations, with step size $\gamma =  \epsilon^{-1}$. Set the termination condition as $\|\nabla f(x)\| \leq \epsilon$. Then $Q^*$ is an $\epsilon$-first-order stationary point that will be reached after $k = L(G(Q^{(0)}) – G(Q^*))\epsilon^{-2}$ iterations, where $L=l^2$ and $l < \infty $ is a bounded parameter.
\end{Theorem}

\begin{Theorem}[Complexity]
	\label{complex_MLOP}
	Given a point-set $P=\{{p_j }\}_{j =1}^J$ sampled near a $d$-dimensional manifold $\mathcal{M} \in \mathbb{R}^n$,  let $Q=\{q_i \}_{i=1}^I$ be a set of points that will provide the desired manifold reconstruction. Then the complexity of the {\rm MLOP} algorithm is $O(nmJ + k I(nm\widehat{I}+\widehat{J}))$, where the number of iterations $k$ is bounded as in Theorem~{\rm \ref{thm:convRate}}, $m \ll n$ is the smaller dimension to which we reduce the dimension of the data, and $\widehat{I}$ and $\widehat{J}$ are the numbers of points in the support of the weight functions $\widehat w_{i,i'}$, $w_{i,j}$   that belong to the $Q$-set and $P$-set, respectively. Thus, the approximation is linear in the ambient dimension $n$, and does not depend on the intrinsic dimension $d$.
\end{Theorem}

\section{Extending the MLOP Method to Manifold Repairing}

The settings for the reconstruction and repairing problem is the following: Given a noisy point-cloud $P=\{p_j\}_{j=1}^J \subset \mathbb{R}^n$ situated near a manifold $\mathcal{M}\subset  \mathbb{R}^n$, of unknown intrinsic dimension $d$ and with incomplete data in a ball $B(c, r)$. We look for a solution in the form of a new point-set $Q=\{q_i\}_{i=1}^I \subset \mathbb{R}^n$ that will replace the given data $P$, be quasi-uniformly distributed, provide a noise-free approximation of $\mathcal{M}$, and will reconstruct the missing information in the given location. This is achieved by leveraging the well-studied weighted $L_1$-median \cite{vardi2000multivariate}, requiring a quasi-uniform distribution of points $q_i \in Q$ and propagating the boundary information inside the hole. 

\forceindent We propose here  the \textit{Repairing Manifold Locally Optimal Projection}(R-MLOP) method, which enhances the MLOP approach (described in Section \ref{MLOP_pre}) and is inspired by a method for local  approximation of functions. In the latter approach, the value of a function at a point $x$ is estimated as a weighted average of the values of the function at nearby points $x_i$:
\begin{eqnarray*} f(x) = \frac{\sum_{i \in I} w_i f(x_i)}{\sum_{i \in I} w_i} \,,
\end{eqnarray*}
where $w_{i} = \exp\!\big\{\!-\|x_i-x\|^2/{h^2}\big\}$ and $h$ is the fill distance of the points $\{x_i\}_{i \in I}$.

\forceindent The proposed R-MLOP method introduces a new term $E_3$ in equation \eqref{eq:1}, which plays the role of a force of attraction, pulling points of the boundary of the hole towards their convex hull. We define this repairing $E_3$ term as
\begin{eqnarray}\label{eq:E3_v2} 
E_3 = \sum_{i' \in I} \bar \tau_{i'}  \left\| q_{i'} \sum_{i \in I} \widehat w_{i,i'} - \sum_{i \in I,\, i\neq i'} \widehat w_{i,i'}q_{i}\right\|^2 \,,
\end{eqnarray}
where $\widehat w_{i,i'} = \exp\!\big\{\!-\|q_i-q_i'\|^2/{h_2^2}\big\}$, $\bar \tau_{i}$ are balancing terms (see below), and the $h_2$ is the expected representative distance of the $Q$-points, as introduced in Definition \ref{def:def2}.

Note that in matrix notation the repairing term $E_3$ can be rewritten as the norm of the graph Laplacian $L$:
\begin{equation} E_3 = \sum_{i' \in I} \bar \tau_{i'} \left\|L(Q-Q_{i'}) \right\|^2\,,
\end{equation}
where $L=D-W$,   $D_{i'i'}=\sum_{i} \widehat w_{i,i'}$ , $W$ is a matrix with entries $\widehat w_{i,i'}$, $\widehat w_{i,i'} = \exp\!\big\{\!-\|q_i-q_i'\|^2/{h_2^2}\big\}$, $Q$ is the matrix with the rows $\{q_{i}\}_{i \in I}$, while $Q_{i'}$ is the matrix with the rows $\{q_{i'}\}$.

We now define the \textit{manifold reconstruction and repairing algorithm} as the minimization of the non-convex function
\begin{equation}\label{eq:Grepair} G(Q) = c_1 E_1(P,Q,T) + c_2 E_2(Q) + c_3 E_3(Q,T)\,, \end{equation}
where 
\begin{equation}
\label{eq:E1_v2}
E_1(P,Q,T) = \sum\limits_{q_i \in Q} (1-\bar \tau_i) \sum\limits_{p_j \in P} \|q_i-p_j\|_{H_{\epsilon}} w_{i,j}\, 
\end{equation}
\begin{equation}
E_2(Q) = \sum\limits_{q_i \in Q} \sum\limits_{q_{i'} \in Q \backslash \{q_i\}} \eta(\|q_i-q_i'\|) \widehat w_{i,i'} \,,
\end{equation}
and $E_3$ is given in \eqref{eq:E3_v2}.

The constants $c_1$, $c_2$ and $c_3$ replace the $\lambda_i$ terms in \eqref{eq:1} and balance the three terms in \eqref{eq:Grepair}. They are calculated as $c_k ={\mathrm{median}}_ {i \in I} (\|\frac{\partial E_k}{\partial q_i}\|)$, for $k=1,2,3$ after the first gradient descent iteration is completed. In addition, $T$ is a vector of weights $\bar \tau_i \in [0,1]$, each assigned to a  point $q_i$, that balance between the forces  of attraction to $P$ and those of  attraction towards the convex hull of the neighboring $Q$-points. It is calculated only in the first iteration, and is based on the distance of the points $q_i$ to the  location of the hole. Specifically, near a hole the value of $\bar \tau_i$ is chosen to be close to $1$ and thus $E_3$ gains more weight, while in regions where no repairing is required $\bar \tau_i$ is small and $E_1$ becomes dominant.

\forceindent Given that the hole is located in a ball $B(c, r)$ (we will discuss how to approximate the parameters $r$ and $c$ later) the weight $\bar \tau_i \in [0,1]$ associated to a point $q_i$ is calculated by the rule
\begin{equation} 
\label{eq:mu2} \bar \tau_i = \frac{\tau_i - \minmath(\tau_i)}{\maxmath(\tau_i) - \minmath(\tau_i)}, \quad i \in I, 
\end{equation}
where 
\begin{equation} 
\label{eq:mu1} \tau_i = \exp\!\left\{-{\|q_i-c\|^2}/{r^2}\right\} \,. \end{equation}

\forceindent The solution to the minimization problem \eqref{eq:Grepair} is found via the gradient descent algorithm:
\begin{equation}
\label{q_i_E3} q_{i'}^{(k+1)}=q_{i'}^{(k)}-\gamma_k \nabla G(q_{i'}^{(k)})\,, 
\end{equation}
where the points ${q_i}^{(0)}$'s are randomly sampled from $P$, and the coefficients $\gamma_k$ are given by formula  \eqref{gamma_k}.\\
The gradient of the R-MLOP cost function $G$ has the expression 
\begin{equation}  
\label{eq:GradG_E3} \nabla G(q_{i'}^{(k)}) = (1-\bar \tau_{i'}) c_1 \sum\limits_{j=1}^J {\left (q_{i'}^{(k)}-p_j\right ) \alpha_j^{i'}}- c_2 \sum\limits_{\substack {i=1 \\ i\neq i'}}^I{ \left (q_{i'}^{(k)} - q_{i}^{(k)} \right ) \beta_i^{i'}} + \bar \tau_{i'} c_3 \frac{\partial E_3}{\partial q_{i'}}\,, \end{equation}
where $\alpha_j^{i'}$ and $\beta_i^{i'}$ are given the formulas \eqref{eq:alpha} and \eqref{eq:beta}, respectively, and the gradient of $E_3$ is given by
\begin{eqnarray} \label{partial_E_3}	
\frac{\partial E_3}{\partial q_{i'}} = 2 \sum_{\substack {i \in I \\ i\neq i'}} {(q_{i'}- q_i) \widehat w_{i,i'}}  \sum_{\substack {i \in I \\ i\neq i'}} {\left( 1-\frac{2}{h_2^2} \|q_{i'} - q_i\|^2 \right) \widehat w_{i,i'} }\,.
\end{eqnarray}

\subsection{Multiple Hole Repair}

In the case of multiple hole repair, it is necessary to modify the $T$ weights in equation \eqref{eq:Grepair} so as to include all the information regarding the holes. Let $B(c_k, r_k)$ be a set of balls, each with incomplete data which need to be amended. Then, the normalized weights $\bar \tau_i^k \in [0,1]$ for a hole $k$ are calculated as 
\begin{equation}
\label{eq:mu2_2} \bar \tau_i^k = \frac{\tau_i^k - \minmath(\tau_i^k)}{\maxmath(\tau_i^k) -\minmath(\tau_i^k)}, \quad i \in I, 
\end{equation}
where 
\begin{equation} \label{eq:mu1_2} \tau_i^k = \exp\!\left\{-{\|q_i-c_k\|^2}/{r_k^2}\right\} \,. 
\end{equation}
We then redefine the normalized weights $\bar \tau_i \in [0,1]$   used in \eqref{eq:Grepair} for the present case of multiple hole repairing  as 
\begin{equation} \bar \tau_i = \frac{\tau_i - \minmath(\tau_i)}{\maxmath(\tau_i) - \minmath(\tau_i)}\,, \end{equation} 
where 
\begin{equation} \tau_i = \prod_{k=1}^K \tau_i^k \,.\end{equation}

This definition of the weights that form $T$  incorporates the information from all the holes in a single coefficient attached to the points $q_i$.

\section{Theoretical Analysis of the Method}

In this section, we  validitate  the proposed method. Specifically, given noisy data with a hole, we investigate whether the proposed method amends the hole, and prove that it creates a noise-free reconstruction of the manifold which is of order $O(C_1  h^2 +C_2 r^2)$. 
In addition, we show that the order of complexity does not change with the new extension, compared to the vanilla MLOP. Last, we propose a method for  hole identification that uses the intrinsic properties of the MLOP mechanism. 

\begin{definition}\label{def:epsilonB}
	Let $\epsilon > 0$ be a small number, and let $\mathcal{M}$ be a $d$-dimensional manifold with a hole $H$ at a location $l \in \R^n$ and with the size bounded by $r$ \textup{(}where $r$ is the smallest radius of a ball that contains the hole\textup{)}. Suppose that the scattered data points $P=\{p_j\}_{j=1}^J \subset \mathbb{R}^n$ were sampled near the manifold $\mathcal{M}$. Let $Q=\{q_i \}_{i=1}^I$ be the sought-for noise-free approximation of $\mathcal{M}$. Let $\bar \tau_i$ be defined as in \eqref{eq:mu2}. Then, the \emph{$\epsilon$-neighborhood of the boundary of the hole} is the set of points $q_i$ such that $\bar \tau_i > 1- \epsilon$ \textup{(}note that $\bar \tau_i \in  [0,1]$\textup{)}. 
\end{definition}

\subsection{Order of Approximation}

\begin{Theorem}[Order of Approximation] \label{thm:repair_approx_order}
	Let $\mathcal{M}$ be a $d$-dimensional manifold in $\mathbb{R}^n$, where $d$ is an unknown intrinsic dimension. Let $H$ be a convex hole in $\mathcal{M}$ bounded by a ball $B(c,r)$. Suppose that the scattered data points $\{{p_j }\}_{j =1}^J$ were sampled from the manifold $\mathcal{M}$ with noise,  $h$ is as in Definition \textup{\ref{def:def2}}, and the $h$-$\rho$ conditions are satisfied with respect to $\mathcal{M}$. Also, let $Q^{(0)}=\{q_i^{(0)} \}_{i=1}^I$ be a set of initial points set sampled from $P$. Then the order of approximation of \textup{R-MLOP} to $\mathcal{M}$ is less than $C_1h^2+C_2r^2$, where the constant $C_1$ depends on the curvature of the manifold outside the hole, and $C_2$ depends on the curvature inside the hole. 
\end{Theorem}
\begin{proof}
	
	In regions far from the boundary of the hole, the $\bar \tau_{i'}$'s,  defined in \eqref{eq:mu2}, are small and therefore the  function  \eqref{eq:Grepair} to be optimized  consists of only the first two terms, $E_1$ and $E_2$. According to Theorem \ref{thm:approx_order}, the order of approximation in these regions is $O(h^2)$, where $h = \max\{h_1, h_2\}$ and $h_1$ and $h_2$ are defined in Definition \textup{\ref{def:def2}} with respect to the  points-sets $P$ and $Q$. 
	
	In regions included in the $\epsilon$-neighborhood of the boundary of the hole, the representative distance between the points is bounded by $2r$ (we assume that $h \leq r$, otherwise there is no hole). In these regions, there are points $q_i \in Q$ which lie  at a distance $O(r)$ from points in the set $P$. Therefore, the overall order of approximation at a new point is a combination of the two, namely $\le C_1 h^2 +C_2 r^2$, with $C_1$, and $C_2$ constants. 
\end{proof} 

\subsection{Method Validation}
%
%

\begin{Theorem}\label{thm:rep_val}
	Let $\mathcal{M}$ be a $d$-dimensional manifold in $\mathbb{R}^n$, where $d$ is an unknown intrinsic dimension. Let $H$ be a convex hole in $\mathcal{M}$ bounded by $B(c, r)$. Suppose that the scattered data points $P=\{{p_j }\}_{j =1}^J$ were sampled near the manifold $\mathcal{M}$, $h_1$ is selected as  in  Definition \textup{\ref{def:def2}}, and the $h$-$\rho$ condition is satisfied with respect to $\mathcal{M}$. Also let $Q^{(0)}=\{q_i^{(0)} \}_{i=1}^I$ be initial set of points sampled from $P$. Then the gradient descent iterations for minimizing \textup{\eqref{eq:Grepair}} produce a quasi-uniformly distributed point-set $Q$  that approximates and reconstructs the manifold, as well as recover missing information inside the hole.
\end{Theorem}
\begin{proof}
	The definition of the \textup{R-MLOP} method \eqref{eq:Grepair} together with \eqref{eq:E1_v2} and \eqref{eq:E3_v2} implies that in regions far from the hole, where $\widehat \tau_i$ are small, the term $E_3$   does not play a significant role. Therefore, in such regions  the \textup{R-MLOP} algorithm behaves like \textup{MLOP}  and, by Theorem \textup{\ref{thm:conv}}, it converges and reconstructs there the manifold. 
	
	\forceindent In an $\epsilon$-neighborhood of the boundary of $H$, where the $\bar \tau_i$'s are close to $1$, the target function \eqref{eq:Grepair} consists of only the last two terms. The term $E_2$ is responsible for the uniform distribution of the points $Q$, and therefore we will consider here only the contribution of the term $E_3$ to hole repairing. 
	
	\forceindent Let us analyze the role of the term $E_3$, and prove that it indeed amends the hole. To this end, we show that at each iteration, the points in the $\epsilon$-neighborhood of the boundary of the hole $H$  move towards the center of $H$. Let $\epsilon > 0$ be a small number, and let $q^{(k)}_{i'}$ be a point in the $\epsilon$-neighborhood of the boundary of   $H$. We would like to prove that as a result of a gradient descent step the point $q^{(k+1)}_{i'} = q^{(k)}_{i'} - \gamma_i \frac{\partial E_3}{\partial q_{i'}}$ moves towards the center of the hole $H$.
	
	The gradient in \eqref {eq:GradG_E3} can be rewritten as 
	\begin{eqnarray} \label{Grad_E3_v2} \frac{\partial E_3}{\partial q_{i'}} = 2 b_{i'}\sum_{i \in I,\, i\neq i'} {(q_{i'}^{(k)}- q_i^{(k)}) \widehat w_{i,i'}} \,, \end{eqnarray} 
	where $b_{i'} = \sum_{ i \in I,\, i\neq i'} {\left( 1-\frac{2}{h_2^2} \|q_{i'}^{(k)} - q_i^{(k)}\|^2 \right) \widehat w_{i,i'} } \,.$ 
	
	In what follows we determine the direction of $\sum_{i \in I,\, i\neq i'} {(q_{i'}^{(k)}- q_i^{(k)}) \widehat w_{i,i'}}$ and the sign of $b_{i'}$.
	In the proof, we rely on a uniform distribution of the  points $p_j$. In order to guarantee this, we first run the vanilla MLOP, which results in a quasi-uniform sampling of the manifold.
	\begin{enumerate}[noitemsep]
		
		\item \textbf{The weighed sum $\vec a = \sum_{ i \in I \, i\neq i'}  {(q_{i'}^{(k)}- q_i^{(k)}) \widehat w_{i,i'}}$ is a vector pointing towards  the center of the hole.} Indeed, the term $\sum_{i \in I,\, i\neq i'} {(q_{i'}^{(k)}- q_i^{(k)}) \widehat w_{i,i'}}$ is a weighted sum of vectors, each directed from $q_i^{(k)}$ to $q_{i'}^{(k)}$ (see Figure \ref{fig:Ammending_hole} left, vectors marked in black). We rewrite this sum according to the direction of the vector $\vec {r}_{i', i} = q_{i'}^{(k)}- q_i^{(k)}$, i.e., towards the center of the hole or not, using a dot product. Let  $\vec v_{i'} = \frac{c-q^{(k)}_{i'}}{\|c-q^{(k)}_{i'}\|}$  be the unit vector originating at a point $q^{(k)}_{i'}$ and pointing in the direction of the center of the hole. Then $\vec a$ can be decomposed into a sum of vectors in the same direction as $\vec v_{i'}$, and a sum of vectors pointing in a different direction:
		\[
		\vec a = \sum_{ \frac{\vec{r}_{i',i}}{\|\vec{r}_{i',i}\| }\boldsymbol{\cdot} \vec v_{i'} \  \geq \  0} {\vec{r}_{i',i} \widehat w_{i,i'}} +\sum_{ \frac{\vec{r}_{i',i}}{\|\vec{r}_{i',i}\|} \boldsymbol{\cdot} \vec v_{i'} \  < \  0} {\vec{r}_{i',i} \widehat w_{i,i'}} \,.
		\]
		
		The hyperplane ${\bf x} \cdot \vec v_{i'} = 0$ separates between the vectors $\vec{r}_{i',i}$ that point in the direction of the center of the hole, from those that are pointing in  different directions. Assuming a uniform distribution of the points, since there are no points within the hole, there are more points that satisfy the inequality $ \frac{\vec{r}_{i',i}}{\|\vec{r}_{i',i}\| }\boldsymbol{\cdot} \vec v_{i'} \  \geq \  0$ then the opposite inequality. Consequently, the vector $\vec a$ points towards the center of the hole.
		
		\item \textbf{The term $b_{i'}$ satisfies the inequality $b_{i'} < 0$.} Let us define $b(r) = \left( 1-\frac{2}{h_2^2} r^2 \right) \exp\{- {r^2}/{h^2}\}$, where $r = \|q_{i'} - q_i\|$. We analyze the behavior of the function $b(r)$, and plot it in Figure \ref{fig:Ammending_hole} (right). In view of   Remark \ref{def:def3}, the term $\exp\{- {r^2}/{h^2}\}$ vanishes for $r > 2 \sqrt 2 h_2$. Furthermore,  note that
		$b(r) \geq 0$ for $r \in [0,\frac{h_2}{\sqrt 2}]$,  
		$b(r) < 0$ for $r \in (\frac{h_2}{\sqrt 2}, 2 \sqrt 2 h_2]$,  
		and $b(r)$ reaches its extrema at $r=0$ and $r = \pm \frac{\sqrt 3 h_2}{\sqrt 2}$. Using these observations, we rewrite the definition of $b_{i'}$ as 
		\begin{equation}
		\label{b_i_def_2} b_{i'} = \sum_{0\leq r \leq \frac{h_2}{\sqrt 2}} {b(r)} - \sum_{\frac{h_2}{\sqrt 2} < r \leq 2 \sqrt 2} {|b(r)|}  \,,\end{equation}
		In what follows we bound the first term in \eqref{b_i_def_2} from above by $c_1$, and the second term from bellow, by $c_2$, and show that $c_1< c_2$.
		
		Let $N_0$ be the estimated number of points within the ball $B_0 = B(q^{(k)}_{i'}, \frac{h_2}{\sqrt 2})$. Then the first term in \eqref {b_i_def_2} can be bounded as $$\sum_{0\leq r \leq \frac{h_2}{\sqrt 2}} {b(r)} \leq N_0 \times \max_{r \in [0,\frac{h_2}{\sqrt 2}]} b(r) \leq N_0 \,.$$
		
		We now turn to bounding the second term in \eqref {b_i_def_2} from bellow. We divide the interval $(\frac{h_2}{\sqrt 2}, 2 \sqrt 2 h_2]$ into four intervals  $A_1 = (\frac{h_2}{\sqrt 2}, A]$,  $A_2 = [A, \frac{\sqrt 3 h_2}{\sqrt 2}]$, $A_3 = [\frac{\sqrt 3 h_2}{\sqrt 2}, C]$ and $A_4 = [C, 2 \sqrt 2 h_2]$, where the points $A$ and $C$ are the centers of the intervals $[\frac{h_2}{\sqrt 2}, \frac{\sqrt 3 h_2}{\sqrt 2}]$ and $[\frac{\sqrt 3 h_2}{\sqrt 2}, 2 \sqrt 2 h_2]$, respectively (i.e., $A = \frac{(1+\sqrt 3)h_2}{2\sqrt 2}$, $C = \frac{(4+\sqrt 3)h_2}{2\sqrt 2}$, see Figure \ref{fig:Ammending_hole} right). Consider four balls centered at the point $q^{(k)}_{i'}$, with the different radii $B_1 = B(q^{(k)}_{i'}, A)$, $B_2 = B(q^{(k)}_{i'}, \frac{\sqrt 3 h_2}{\sqrt 2})$, $B_3 = B(q^{(k)}_{i'}, C)$ and $B_4 = B(q^{(k)}_{i'}, 2 \sqrt 2 h_2)$, and  estimate the number of points in these balls by $N_1$, $N_2$, $N_3$ and $N_4$, respectively.

		With this notation,  
		\begin{multline}
		\label{second_term_bound}
		\sum_{\frac{h_2}{\sqrt 2} < r \leq 2 \sqrt 2} {|b(r)|} \geq \sum_{r\in A_2} {|b(r)|} + \sum_{r \in A_3} {|b(r)|} 
		\geq  N'_2 \min_{r \in A_2} {|b(r)|} +  N'_3 \min_{r \in A_3} {|b(r)|} \\ \geq N'_2\, b(A) +  N'_3\, b(C) \geq 0.3 N'_2  + 0.1 N'_3 
		\end{multline} 
		where $N'_2$  and $N'_3$ are  the estimates of the numbers of points in $B_2 \backslash B_1$ and in $B_3 \backslash B_2$, respectively,  i.e., $N_2  = N_2 - N_1$ and $N'_3 = N_3 - N_2$.
		
		In order to evaluate $N'_1$  and $N'_2$ in terms of $N_0$, we assume that the proportion of the number of points in a support does not change when the radius changes. Let $V_i$ denote  the volume of the ball  $B_i$, $i=0,1,2,3$. Then, relying on the proportion consistency assumption, we estimate $N_1$ in terms of the ratio $\frac{V_0}{V_1}= \frac{N_0}{N_1}$, where the volume of a ball of radius $\frac{h_2}{\sqrt 2}$ in $\mathbb{R}^d$ is $V_0 = \pi^{d/2} (\frac{h_2}{\sqrt 2})^d/c(d)$, and the volume of a ball of radius $R$ is $V_1=\pi^{d/2} {R}^d/\Gamma(d)$  where $\Gamma$ is  the Euler gamma function. Thus, $N_1= \big(\frac{(1+\sqrt 3)h_2}{2\sqrt 2}\big)^d N_0$. Similarly, $N_2= \sqrt 3 ^d N_0$, and  $N_3= \big(\frac{(4+\sqrt 3)h_2}{2\sqrt 2}\big)^d N_0$. Finally, after calculating  $N'_2$ and $N'_3$, and substituting the results in  \eqref{second_term_bound}, we find  that the second term is larger than $c_2 = c_3 N_0$, where $c_3 >1$. Consequently, the second term in \eqref {b_i_def_2} is dominant and $b_{i'} < 0$.
	\end{enumerate}
	
	We conclude that the weighed sum $\sum_{\substack {i \in I \\ i\neq i'}} {(q_{i'}- q_i) \widehat w_{i,i'}}$ is directed towards the hole, $b$ is negative, and so $\frac{\partial E_3}{\partial q_{i'}}$  points outside the hole. In addition, our numerical investigation showed that adopting a small step size of $0.25 \gamma_k$, where $\gamma_k$ is given by formula \ref{gamma_k}, results in a positive $\gamma_k$ during the gradient descent iterations for points in the $\epsilon$-neighborhood of the boundary of the hole.  It follow that the point  $q^{(k)}_{i'}$   in \eqref{q_i_E3} moves towards the hole and recovers missing information.
\end{proof}
\vspace{-8mm}
\begin{figure}[H]
	\centering
	\label{fig:Ammending_hole}\includegraphics[width=\textwidth,height=4.5cm,keepaspectratio]{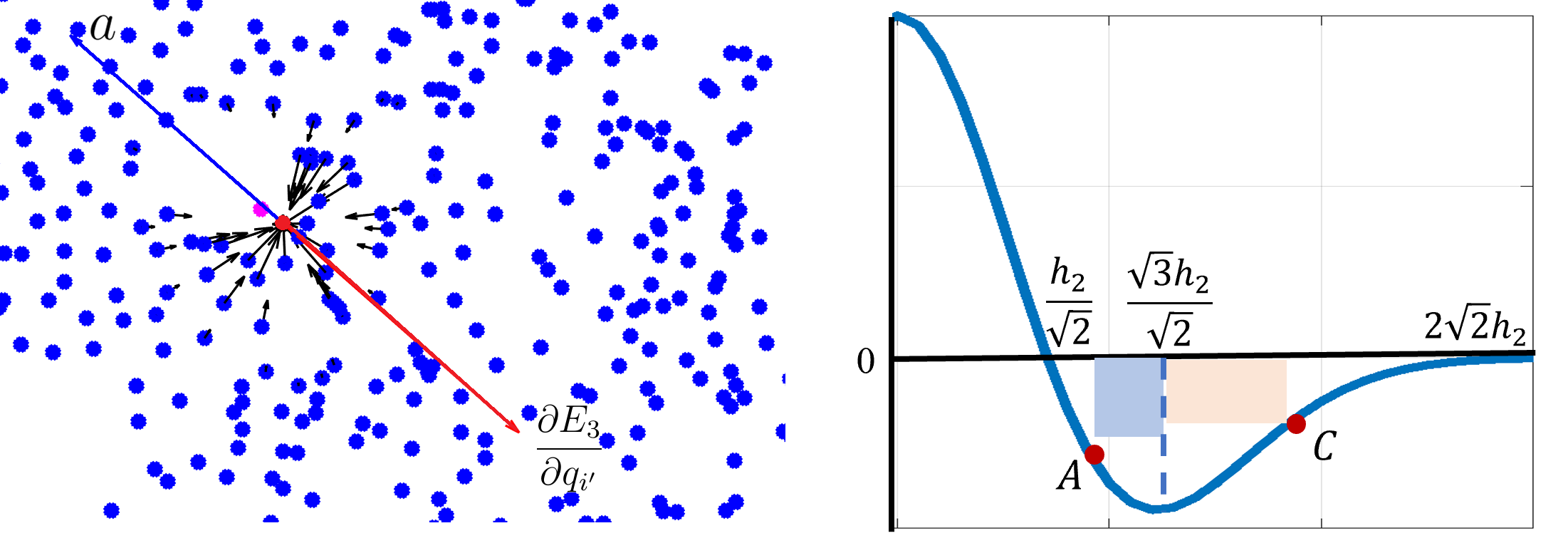}
	\caption{Illustration of the two stages in the validation of the R-MLOP method. Left: Analysis of the direction of movement of a point  $q^{(k)}_{i'}$ (in red) after an amendment step (the magenta point) on a actual case example. The forces by means of which the neighboring points ${q_i}$ act on the current point are in shown in black. The weighted sum in the direction $\vec a$ is shown in blue, while the gradient of the term $E_3$ is shown in red. Thus, the point $q^{(k)}_{i'}$ moves in a direction opposite to that of the gradient, towards the center of the hole. Right: A plot of the function $b(r)$, along with important key points which aid in bounding $b_{i'} <0$. }
	\label{fig:Ammending_hole}
\end{figure}

\subsection{Complexity of the R-MLOP Algorithm}

The R-MLOP algorithm is based on the MLOP algorithm, with the addition of the term $E_3$. As specified in Theorem \ref{complex_MLOP}, the complexity of the MLOP algorithm  is $O(nmJ + k I(nm\widehat{I}+\widehat{J}))$, where $n$ is the dimension of the data and $m$ is the dimension to which one reduces the dimension of the data for the procedure of calculating the norm ($m\ll n$). Thus, we need to evaluate the complexity of the term $E_3$. The gradient of this term involves the differences $q_{i'} - q_i$ between a given point $q_{i'}$ and points $q_i \in Q$. This calculation is performed for each iteration. Thus, the total complexity of the R-MLOP algorithm is $O(nmJ + k I(nm\widehat{I} + \widehat{I} +\widehat{J}))$.

\begin{corollary}
	Given a set of points $P=\{{p_j }\}_{j =1}^J$ sampled near a $d$-dimensional manifold $\mathcal{M} \in \mathbb{R}^n$, let $Q=\{q_i \}_{i=1}^I$ be a set of points which will provide the desired the manifold reconstruction and repairing. Then the complexity of the \textup{R-MLOP} algorithm is $O(nmJ + k I((nm+1)\widehat{I}+\widehat{J}))$, where the number of iterations $k$ is bounded as in Theorem \textup{~\ref{thm:convRate}}, $m$ is the smaller dimension to which one reduces the dimension of the data, $m \ll n$, and $\widehat{I}$ and $\widehat{J}$ are the numbers of points from the sets $Q$ and $P$, respectively, in the support of the weight function $w_{i,j}$. Therefore, the approximation is linear in the ambient dimension $n$, and does not depend on the intrinsic dimension $d$.
\end{corollary}

\section{Approximating the Locations of the Holes and Their Volume}

Hole identification is an ill-posed problem, since usually the geometry of the manifold is unknown, so it is hard to know whether a hole does actually exist,  or the manifold was sampled poorly. Therefore, the best scenario for fixing the missing information is when one can manually identify the holes that need to be amended. However, in real-life scenarios, this information is not always available, and thus estimating the location and radii of holes is required.

\forceindent In this subsection, we propose a method that relies on point density considerations to identify the boundary of a hole. It is reasonable to assume that near the  boundary of a hole the density of the points drops. However, low density does not characterize only regions close to a hole, but also can stem from non-uniform sampling. In addition, for an open manifold, low-density values can also characterize the boundary of the manifold. As described in  Section \ref{repairing_intro}, the problem of hole identification was addressed with various methods, with the most prominent one being, for the point-cloud case,  the method based on $k$-nearest neighbors considerations, which is closely related to density of points. However, the challenge of dealing with manifold boundaries was not taken into account.

\forceindent In what follows we propose a procedure for identifying the boundary of a hole in a manifold, while taking into account the challenges mentioned above. Given scatter data $P=\{p_j\}_{j=1}^J$ sampled from a manifold $\mathcal{M}$ that satisfy the $h$-$\rho$ condition with respect to $\mathcal{M}$, we propose addressing the hole identification task in several phases, each dealing with another challenge caused by low density:

\begin{enumerate}[noitemsep]
	\item Quasi-uniform sampling of the manifold using the vanilla MLOP method.
	\item Identification of the boundary of the manifold.
	\item Identification of the boundary of the hole.
	\item Estimation of the location of the hole and its volume.
\end{enumerate}

First, we apply the MLOP algorithm as a pre-step in order to overcome the case of non-uniform sampling of the $P$-point data. Next, we classify the quasi-uniform point-set $Q$  produced by the MLOP algorithm to be either: a) the boundary of the manifold; b) the boundary of the hole; c) neither of them. Let $h_{0,2}$ be the fill distance of the $Q$-points (see Definition \ref{def:def1}), and let $\rho_2$ be the density of $Q$, which satisfies the inequality \eqref{def:def4}, i.e., $\#\{Q \cap \bar{B}(y,kh_{0,2})\}\leq \rho_2 k^n, \quad k \geq 1, \quad y\in Q$. Note that on the manifold boundary of the manifold the number of points does not grow at the same rate as $k$. Thus, in order to identify the boundary of the manifold one needs to analyze the change in the number of points for $kh$ for $k=1, 2$, and select the points $q_i$ such that their change in $\#\{Q \cap \bar{B}(y,kh_{0,2})\}$ is insignificant. Last, the points on the boundary of the hole are identified as the ones with a low number of points, that do not belong to the manifold boundary. Algorithm \ref{alg:Alg_hole_id} summarizes this process.
%

\begin{algorithm} [H]
	\caption{Estimating the Location of the Hole}
	\label{alg:Alg_hole_id}
	\begin{algorithmic}[1]
		\State {\bfseries Input:} $P=\{p_j\}_{j=1}^J \subset \mathbb{R}^n$
		\State {\bfseries Output:} Hole parameters $r$, $c$
		\State {Create a quasi-uniform point-set $Q$ via vanilla MLOP.}
		\State {Identify points on the manifold boundary by calculating the number of points for $h$, $2h$, and use the ratio $a_i= \frac{\#\{Q \cap \bar{B}(y,2h_{0,2})\}}{\#\{Q \cap \bar{B}(y,h_{0,2})\}}, \quad y \in Q$ to identify the boundary points $q_i$ such that $a_i < {\rm median}_i \{a_i\}$.}
		\State {Identify the points on the boundary of the hole, such that the number of points at this points is low (and they does not belong to the manifold boundary.}
		\State {Clean outliers in the set of boundary points.}
		\State {Approximate the location of the hole and its volume. Estimate the radius $r$ of the hole as half of the maximum of the distance between the boundary points, and the center $c$ of the hole as the center of mass of the boundary points.}
	\end{algorithmic}
\end{algorithm}


\section{Numerical Examples}

%
%


In this subsection, we demonstrate the efficiency of our method on 3D surfaces as well as on manifolds in higher dimension, for single and multiple hole repairing. The numerical setup was usually built by sampling known manifold data, and artificially creating holes in it. It should be noted that in all of the examples below we relied on the fact that the location of the hole and its radius were known.

\subsubsection*{Data Repairing in Low-Dimensional Space}

We start by demonstrating our data completion method on 3D scenes of the bunny and the dragon taken from the Stanford Scanning Repository \cite{levoy2005stanford}. We loaded the two published models, and randomly sampled 1000 points which served as the initial $P$ points. Next, a hole was artificially created by removing points from a chosen location (the hole in the bunny was in its neck, while the hole in the dragon was in its head), which resulted in about 950 points. Later, a subset of 350 points was sampled to construct the $Q$ set. In addition, we slightly increased the density of the $Q$ points near the boundary of the hole. An illustration of the initial settings for the repairing algorithm is provided in Figure \ref{fig:repair_bunny} (A), Figure \ref{fig:repair_dragon_head_hole} (A). 

\forceindent In our numerical experiment, we compared the result of applying the vanilla MLOP with the result of R-MLOP. First, we applied the plain MLOP algorithm on our data, which resulted in a quasi-uniform sampling (Figure \ref{fig:repair_bunny} (B)). As expected, the reconstruction maintained its proximity to the $P$ points and did not recover the missing information. Next, using the information on the hole, we calculated the proximity coefficient $T$ in \eqref{eq:mu2} (its values are presented in Figure \ref{fig:repair_dragon_head_hole} (B)). Finally, we executed the repairing algorithm described above. The amended result after 30 iterations can be seen in Figure \ref{fig:repair_bunny} (C), Figure \ref{fig:repair_dragon_head_hole}, (C). One can see that the existing holes were repaired successfully with uniform sampling.
\vspace{-12mm}

\begin{figure}[H]
	\centering
	\captionsetup[subfloat]{farskip=0pt,captionskip=0pt, aboveskip=0pt}
	\subfloat[][]{ \includegraphics[width=0.34\textwidth]{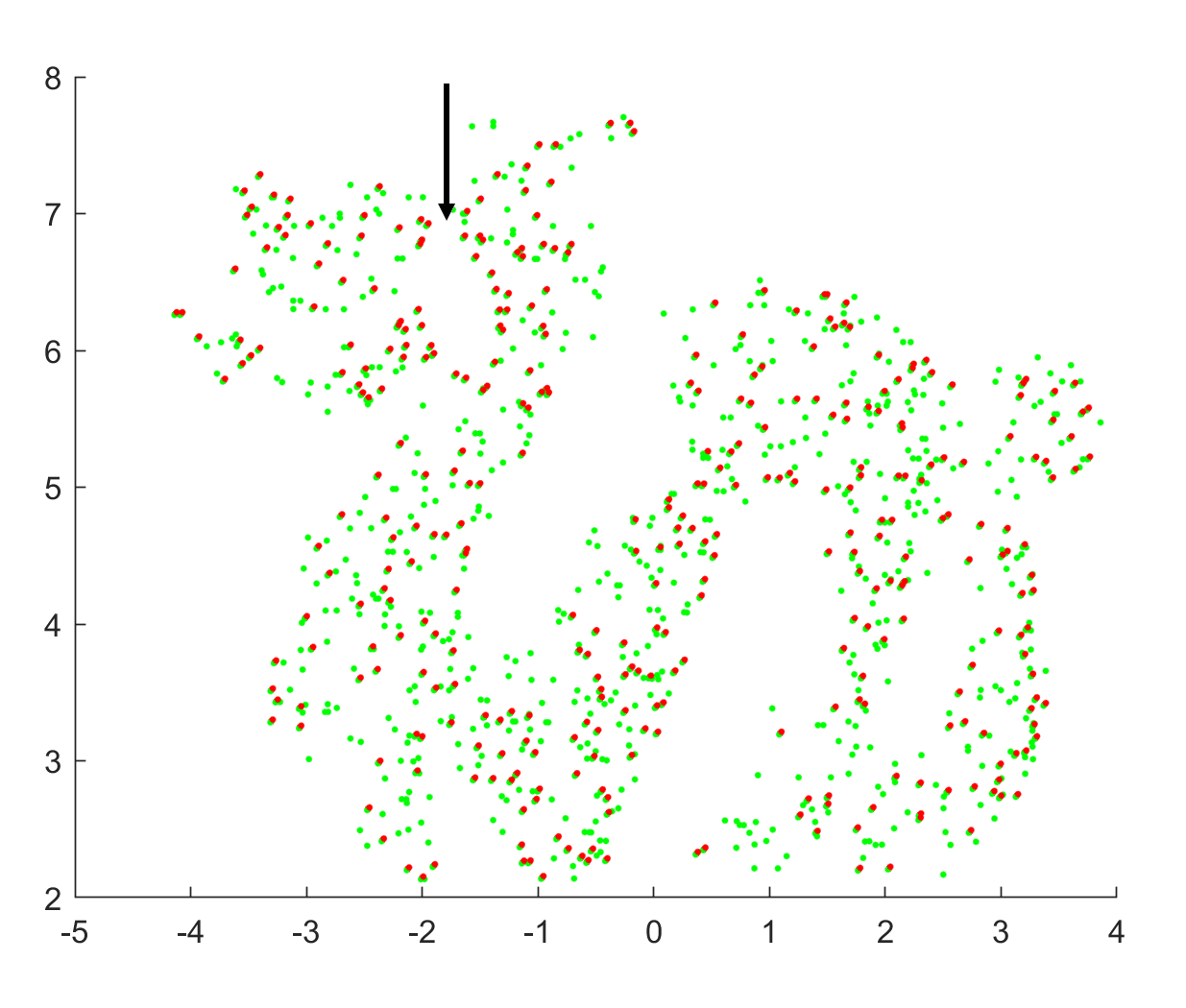} } \hspace{-2em}
	\subfloat[][]{ \includegraphics[width=0.34\textwidth]{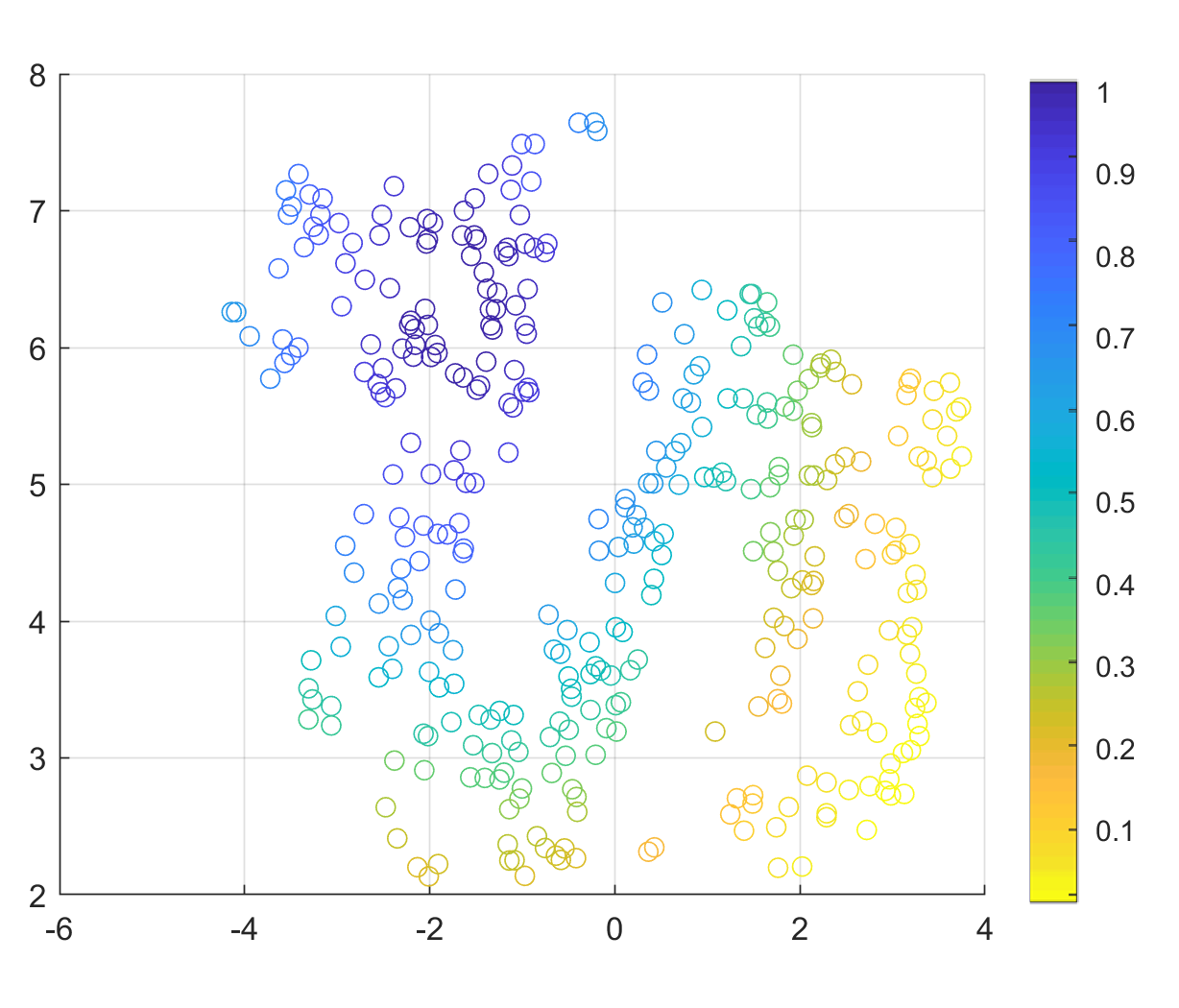} } \hspace{-1.5em}
	\subfloat[][]{ \includegraphics[width=0.34\textwidth]{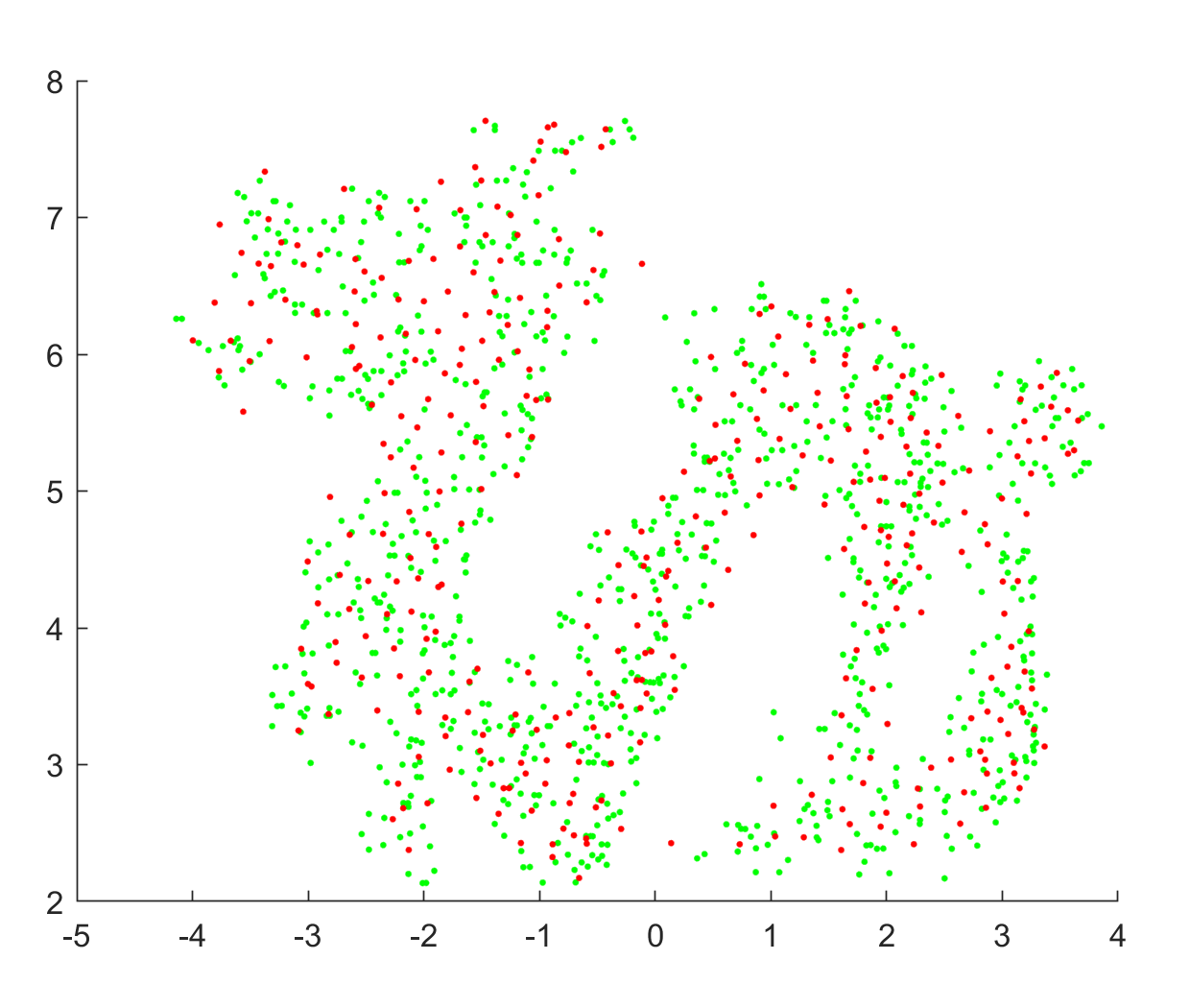} } 
	\caption{Amending the scattered data for the Stanford dragon. (A) The initial scattered data sampled from the dragon model, with a hole artificially created in its head (the initial $P$-points in green, the initial $Q$-points in red). (B) The weights $T$ of the hole, which are closer to $1$ near the hole location. (C) The result produced by the R-MLOP algorithm.}
	\label{fig:repair_dragon_head_hole}
\end{figure}

\subsubsection*{Multiple Holes Repair}

In the next example, we applied the multiple holes repair methodology to the Stanford dragon, which was sampled with $P$ and $Q$ as described above. Two holes locations were selected, one in the dragon head and the other in its tail, and points around them were removed (see Figure \ref{fig:repair_dragon_two_holes_v2}, (A)). Subsequently, we calculated the enhanced proximity weights $T$; they are presented in Figure \ref{fig:repair_dragon_two_holes_v2}, (B). The result produced by the surface amending algorithm is presented in Figure \ref{fig:repair_dragon_two_holes_v2}, (C). One can see that the missing information on the two holes was successfully recovered with quasi-uniform sampling.

\begin{figure}[H]
	\centering
	\captionsetup[subfloat]{farskip=0pt,captionskip=0pt, aboveskip=0pt}
	\subfloat[][]{ \includegraphics[width=0.34\textwidth]{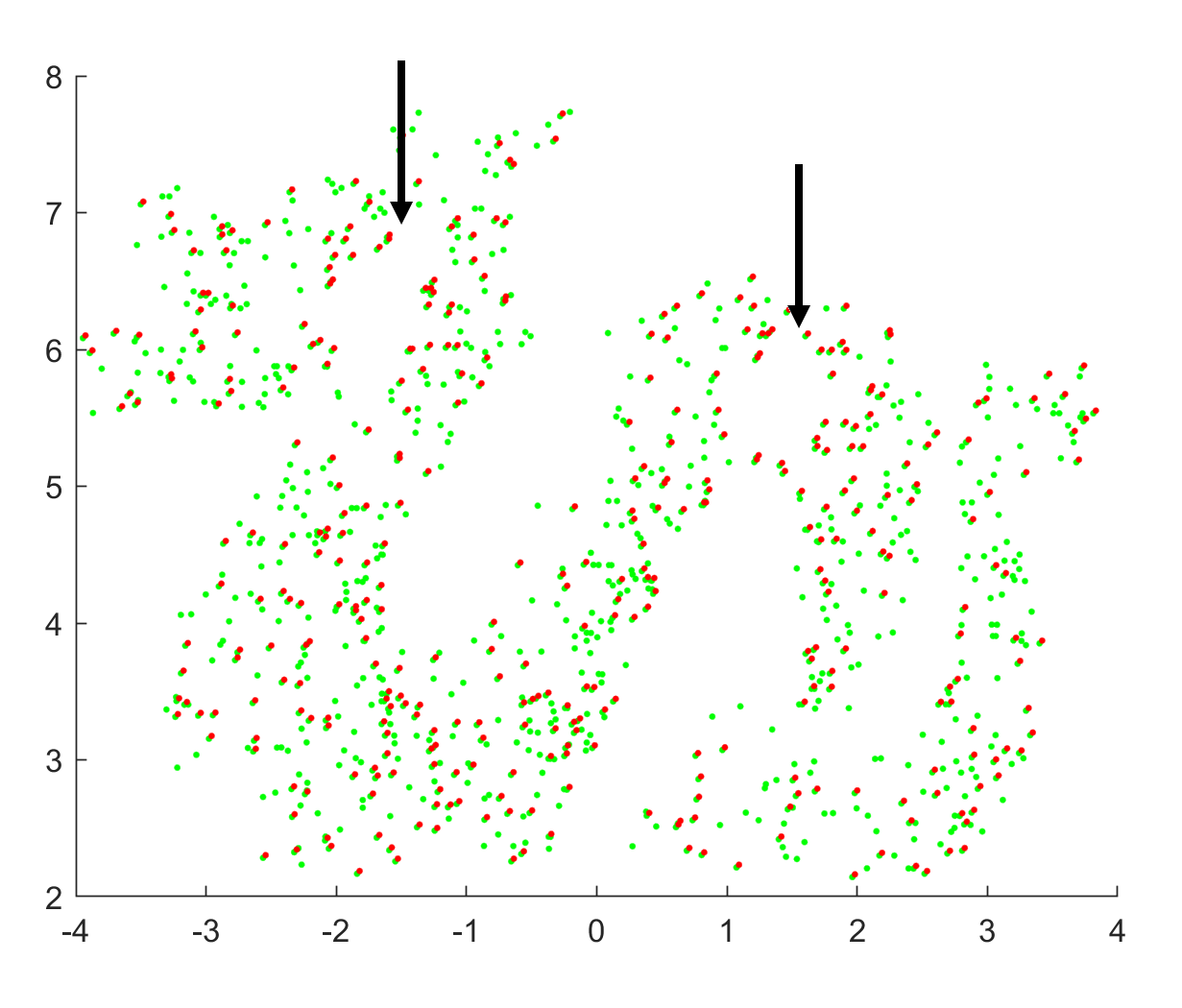} } \hspace{-2em}
	\subfloat[][]{ \includegraphics[width=0.34\textwidth]{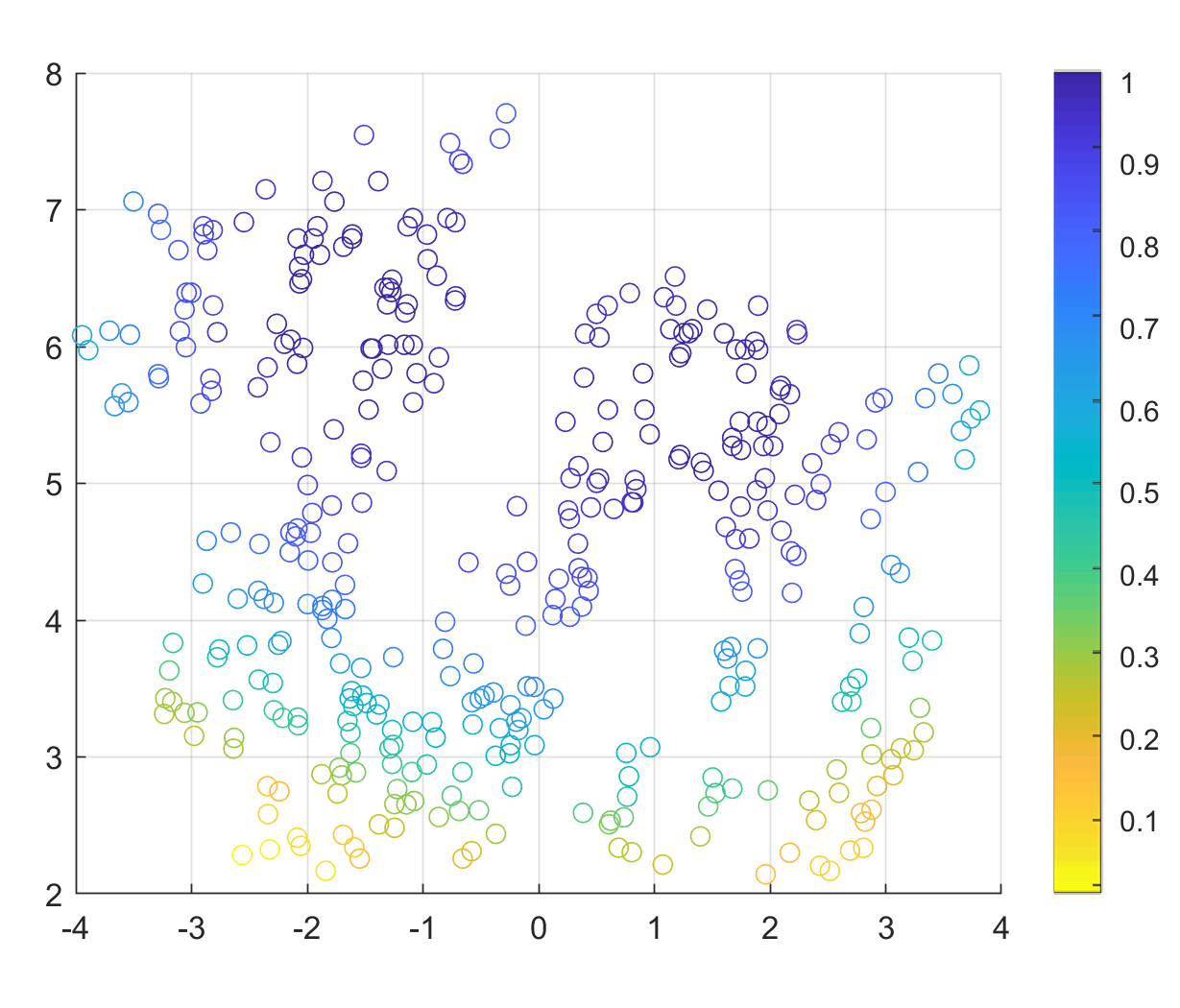} } \hspace{-1.5em}
	\subfloat[][]{ \includegraphics[width=0.34\textwidth]{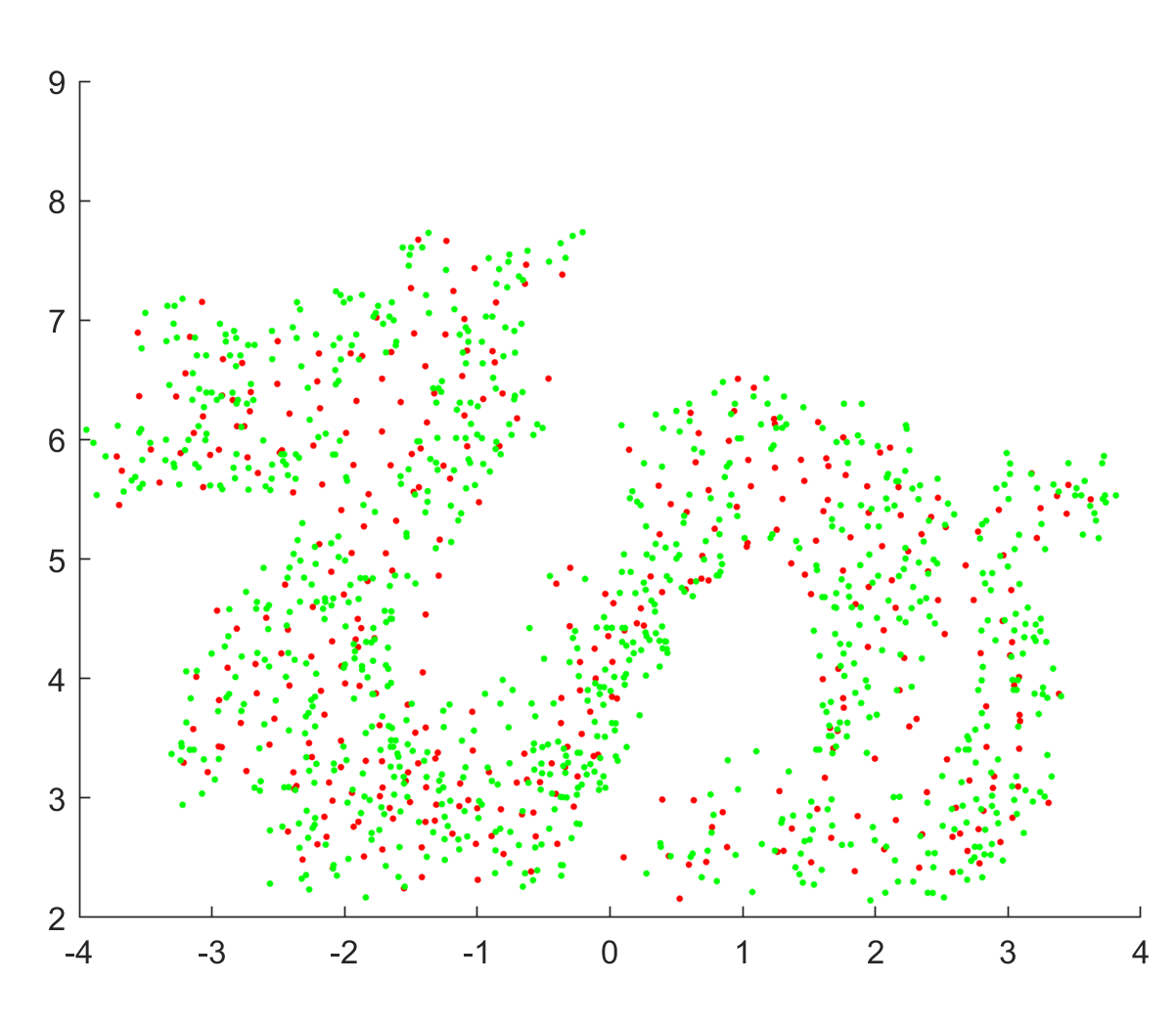} } 
	\caption{Amending the scattered data for the Stanford dragon. (A) The initial scattered data sampled from the dragon model, with two holes artificially created, one in its head and another in its tail marked with arrows (the initial $P$-points in green, the initial $Q$-points in red). (B) The weights of the hole $T$ which are closer to $1$ near the hole location. (C) The result produced by the R-MLOP algorithm.}
	\label{fig:repair_dragon_two_holes_v2}
\end{figure}

\subsubsection*{Manifold Repairing in High-Dimensional Space}
In this subsection, we demonstrate the data completion method on several examples of a low-dimensional manifold embedded in a high-dimensional space. Specifically, we embedded a two-dimensional and a six-dimensional cylindrical structure in $\R^{60}$, as well as a cone structure with multiple holes, also embedded in $\R^{60}$. In all of cases we artificially created a hole in the data around a certain point (Figure \ref{fig:repair_cylinder}, Figure \ref{fig:repair_7Dcylinder} and Figure \ref{fig:repair_cone_two_holes_v1}, (A)). We then applied the R-MLOP method for data completion. The details of the data creation are described below.

\label{sec:Cylindr_example}
\forceindent First, we embedded a two-dimensional cylindrical structure in a 60-dimensional linear space. We sampled the structure using the parametrization
\begin{eqnarray*}p=t v_1+\frac{R}{\sqrt 2}(\cos(u)v_2+  \sin(u)v_3)\,,\end{eqnarray*}
where $v_1=[1,1,1,1,1, \dotsc ,1]$, $v_2=[0,1,-1,0,0, \dotsc ,0], v_3=[1,0,0,-1,0, \dotsc ,0]$  
$(v_1,v_2,v_3 \in \R^{60})$, $t \in [0,2]$ and $u \in [0.1 \pi,1.5 \pi]$. Using this representation, $790$ uniformly distributed (in parameter space) points were sampled with uniformly distributed noise (i.e., $U(-0.1, 0.1)$). The initial $Q$ set was constructed by randomly sampling $230$ points (Figure \ref{fig:repair_cylinder} (A)). Next, we applied the plain MLOP algorithm on our data, which resulted in a quasi-uniform sampling (Figure \ref{fig:repair_cylinder} (B))). As expected, the reconstruction maintained its proximity to the $P$ points, and did not recover the missing information. Finally, we executed the R-MLOP algorithm, and after 70 iterations the manifold was amended with quasi-uniformly sampling (Figure \ref{fig:repair_cylinder} (C)). 

\vspace{-10mm}
\begin{figure}[H]
	\centering
	\captionsetup[subfloat]{farskip=0pt,captionskip=0pt, aboveskip=0pt}
	\subfloat[][]{ \includegraphics[width=0.34\textwidth]{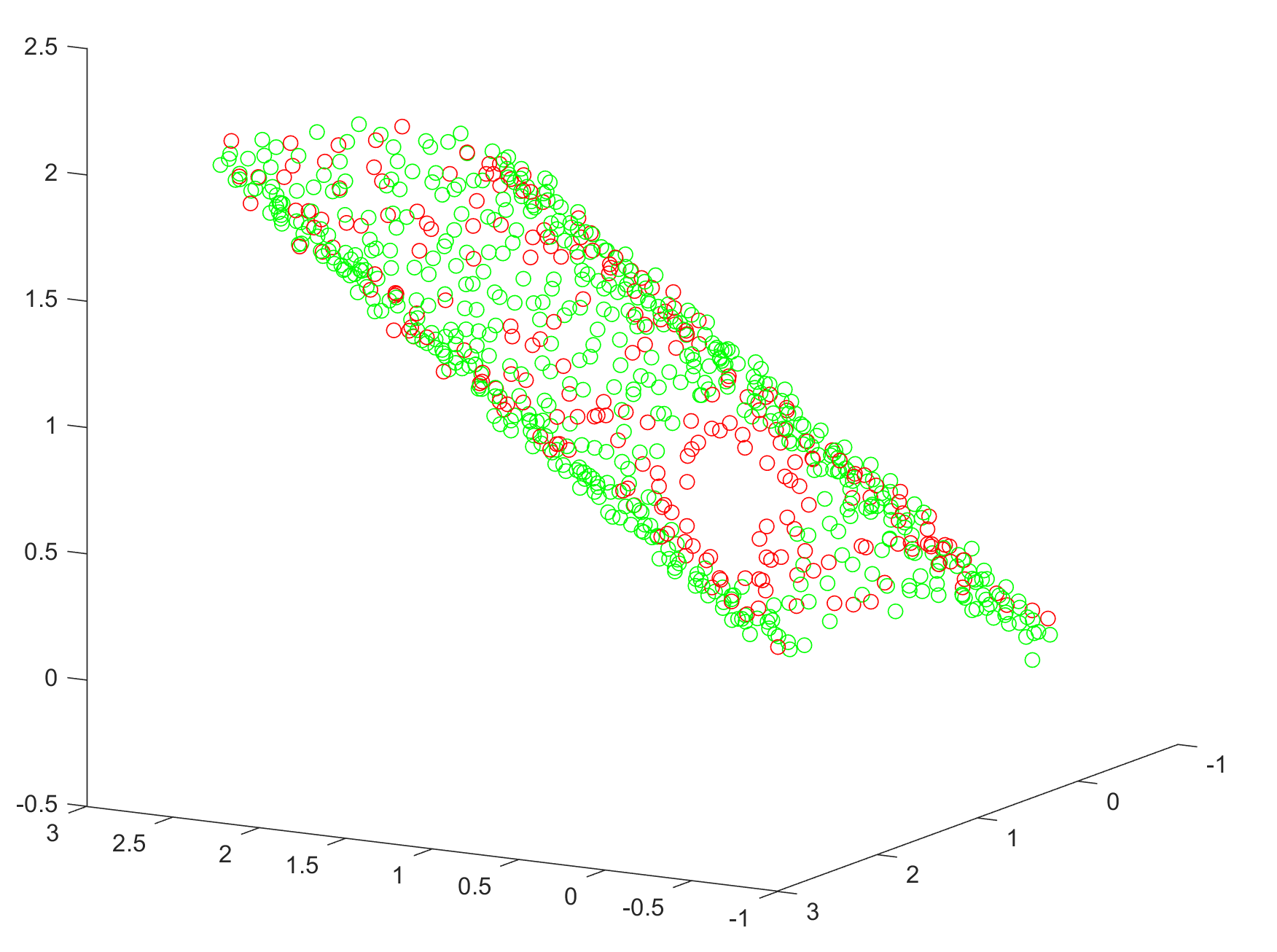} } \hspace{-2em}
	\subfloat[][]{ \includegraphics[width=0.34\textwidth]{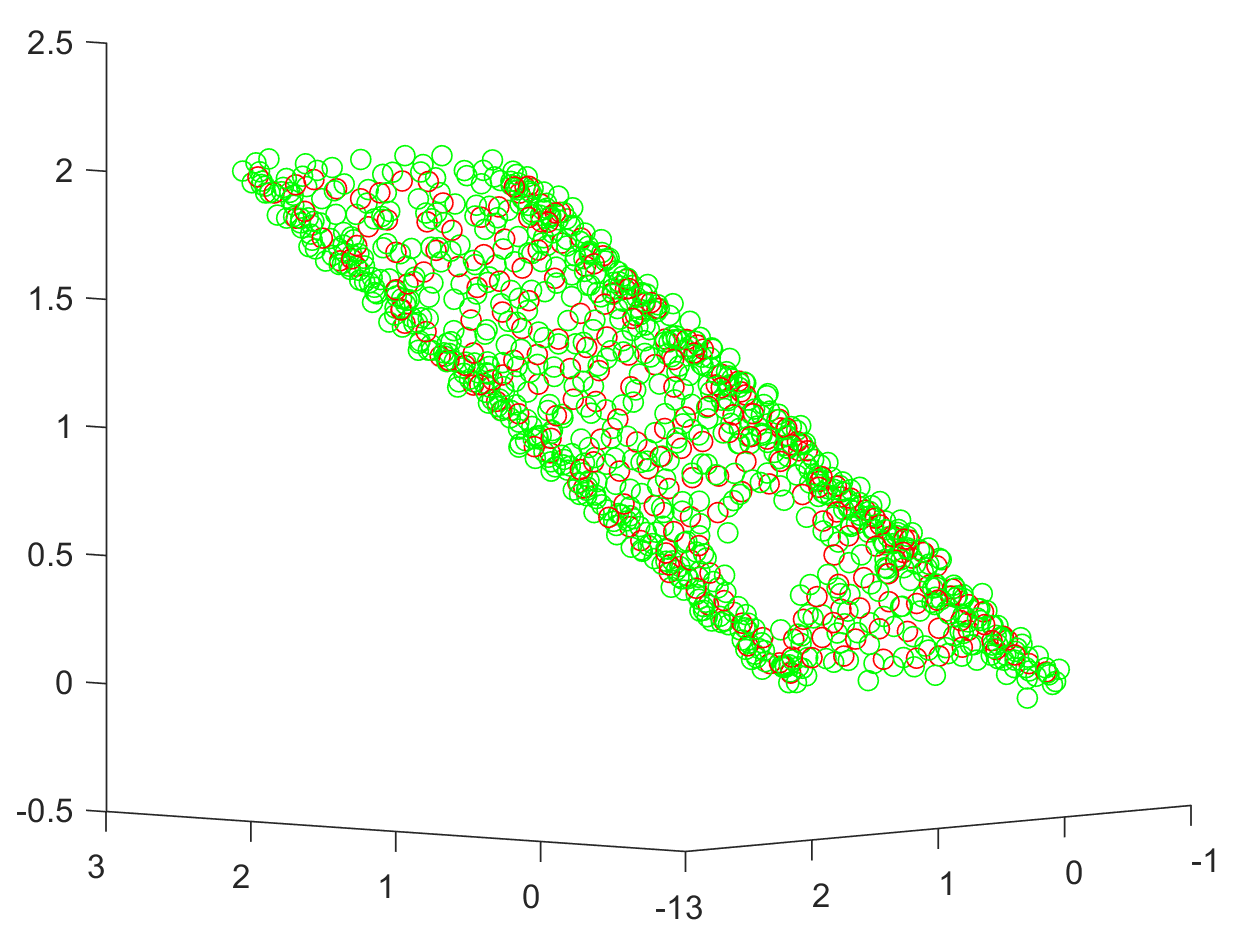} } \hspace{-2em}
	\subfloat[][]{ \includegraphics[width=0.34\textwidth]{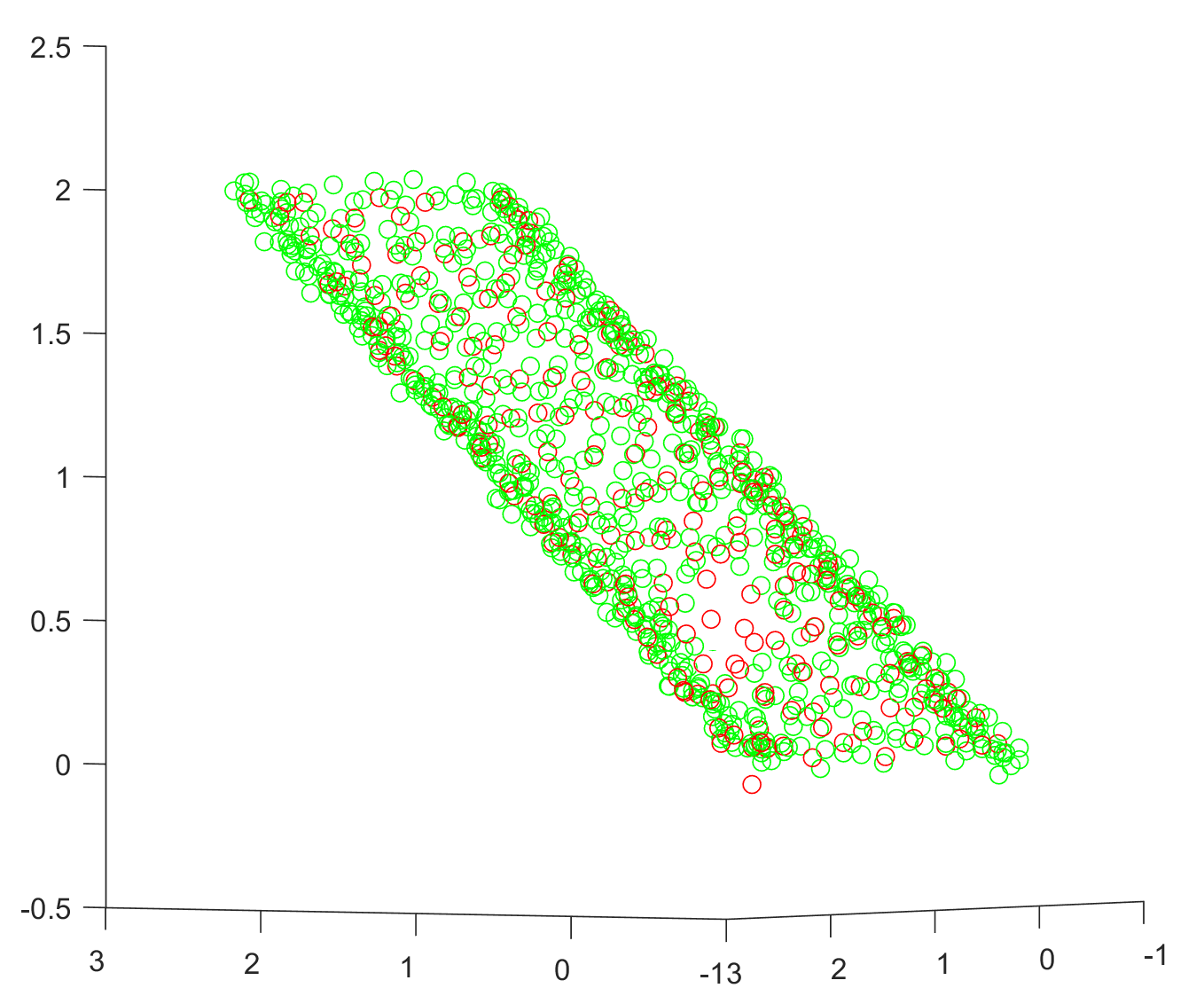} } 
	\caption{Cylindrical structure embedded into a 60-dimensional space. The first three coordinates of the point-set are shown. (A) Scattered data with uniformly distributed noise $U(-0.1; 0.1)$ (green), and the initial point-set $Q^{(0)}$ (red), with an artificial hole created. (B) The point-set generated by the MLOP algorithm. (C) The result produced by the R-MLOP algorithm.}
	\label{fig:repair_cylinder}
\end{figure}

\subsection*{Six-dimensional cylindrical structure}
\label{sec:10dimCylinder_example_1}

Next, we tested our method on higher-dimensional manifolds  by utilizing an $n$-sphere to generate an $(n+1)$-dimensional cylinder (in the example of the two-dimensional cylinder, we used a circle  to generate the structure). Specifically, we utilized a five-dimensional sphere to build a six-dimensional manifold, using the parameterization
\begin{align*}
x_1 = R \cos(u_1) \,,\quad
x_2 = R \sin(u_1)\cos(u_2),\quad 
\ldots,\quad 
x_{6} = R \sin(u_1)\sin(u_2)\cdots\sin(u_5)\sin(u_6)\,. 
\end{align*}
We then embedded the sampled data in a 60-dimensional space by the parametrization
\begin{eqnarray}
p=t v_0 + R^2 [x_1,x_2, x_3, x_4, x_5, x_6, 0, \dotsc ,0]\,,
\end{eqnarray}
where $R = 1.5$, $t \in [0,2]$, $u_i \in [0.1 \pi, 0.6 \pi]$, and $v_0 \in \R^{60}$ is a vector with 1's in positions $1,...,d+1$ and 0 in the remaining positions. We randomly sampled the $P$-points from the six-dimensional cylindrical structure and artificially created a hole in a known location and of known size. We embedded the sampled data in  a 60-dimensional space. This process resulted in $1180$ $P$-points. Next, uniformly distributed noise $U(-0.2; 0.2)$ was added, and then a subset of 460 points was sampled to construct the $Q$ set. To avoid trying to visualize a six-dimensional manifold, we plot here the cross-section of the cylindrical structure in three dimensions.

\forceindent Our numerical experiment included several executions. First, we applied the plain MLOP algorithm on our data, which resulted in a quasi-uniform sampling (Figure \ref{fig:repair_7Dcylinder} (B)). As expected, the reconstruction maintained its proximity to the $P$ points, and did not recover the missing information. Next, using the hole information we calculated the proximity coefficient $T$ in \eqref{eq:mu2} (the resulting values are presented in Figure \ref{fig:repair_7Dcylinder} (C)). Finally, we executed the repairing algorithm described above. The amending result after 100 iterations is shown in Figure \ref{fig:repair_7Dcylinder} (D). As one can see, the holes were repaired successfully with the high-dimensional cylindrical structure with uniform sampling.

\vspace{-10mm}
\begin{figure}[H]
	\centering
	\captionsetup[subfloat]{farskip=0pt,captionskip=0pt, aboveskip=0pt}
	\subfloat[][]{ \includegraphics[width=0.5\textwidth]{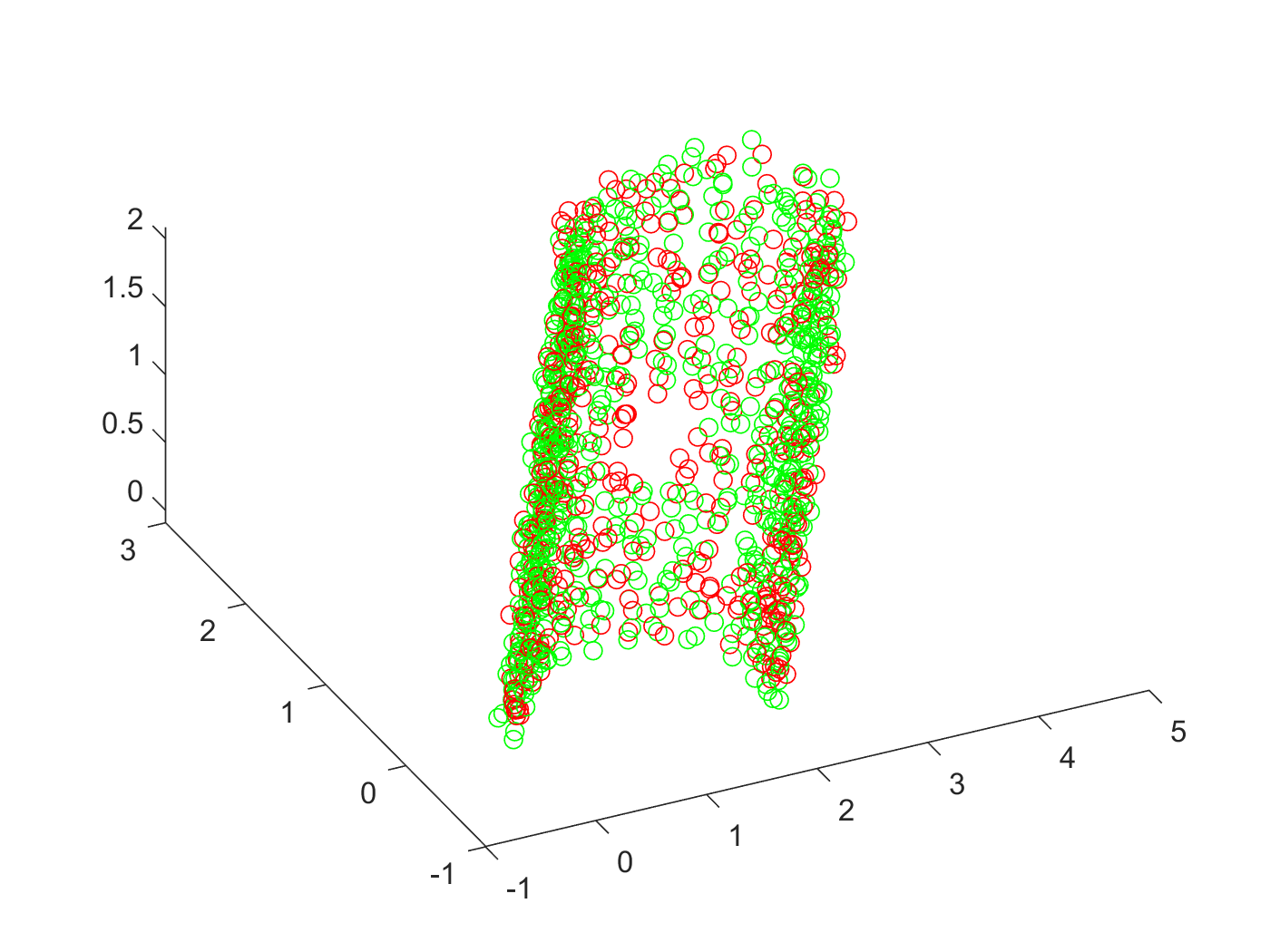} } \hspace{-2em}
	\subfloat[][]{ \includegraphics[width=0.5\textwidth]{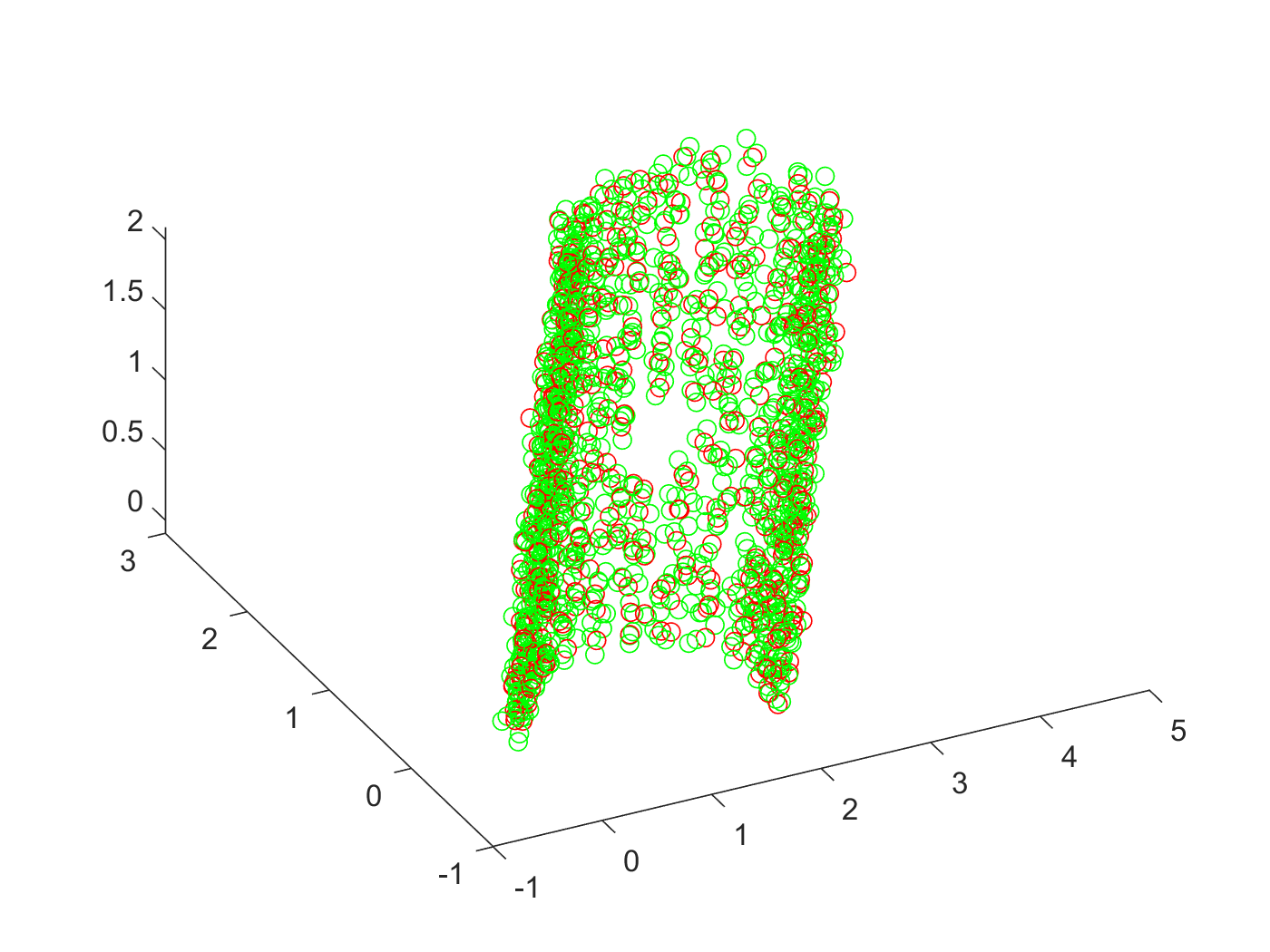} } \\[-0.7ex]
	\subfloat[][]{ \includegraphics[width=0.5\textwidth]{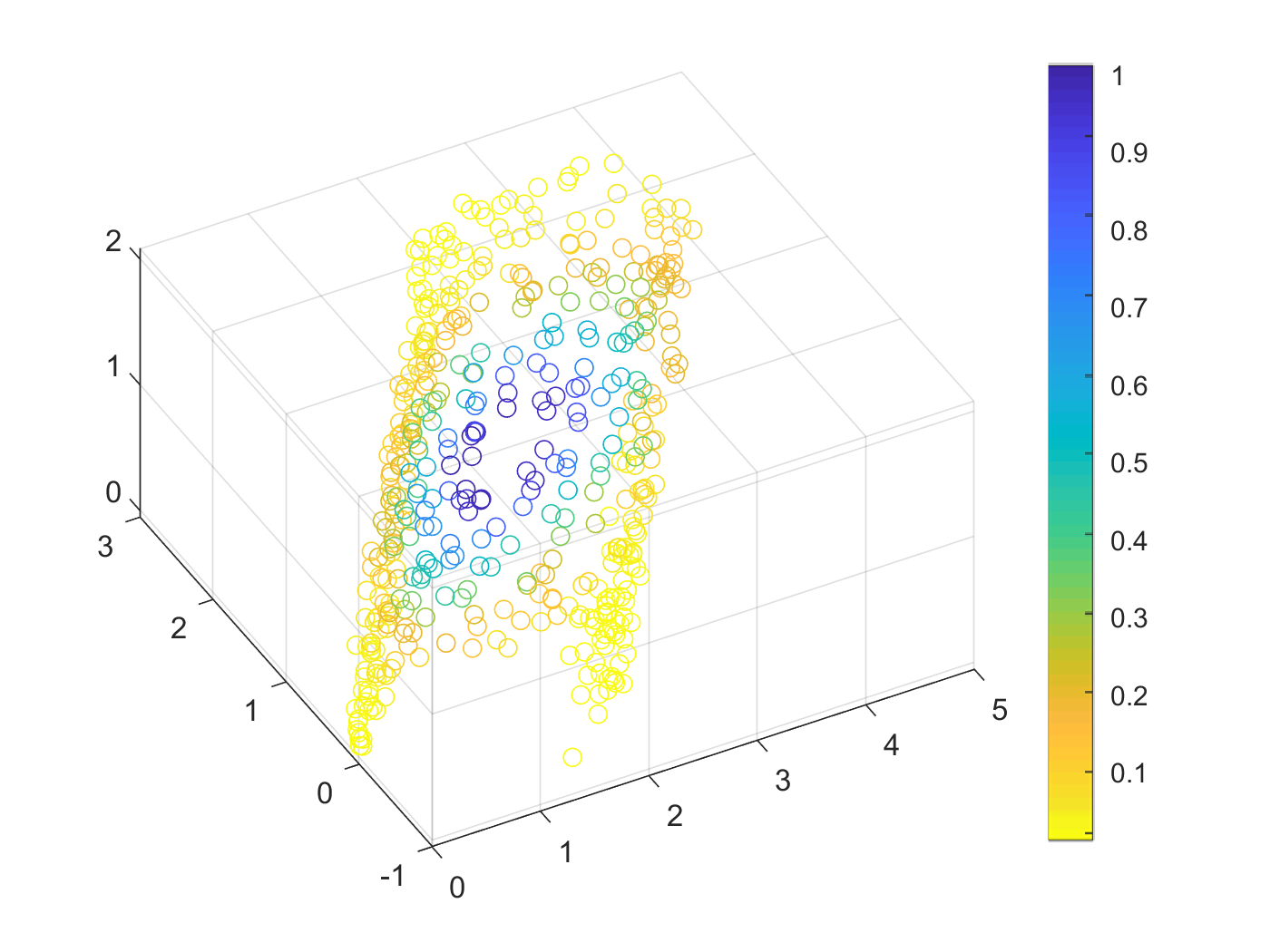} } \hspace{-2em}
	\subfloat[][]{\includegraphics[width=0.5\textwidth]{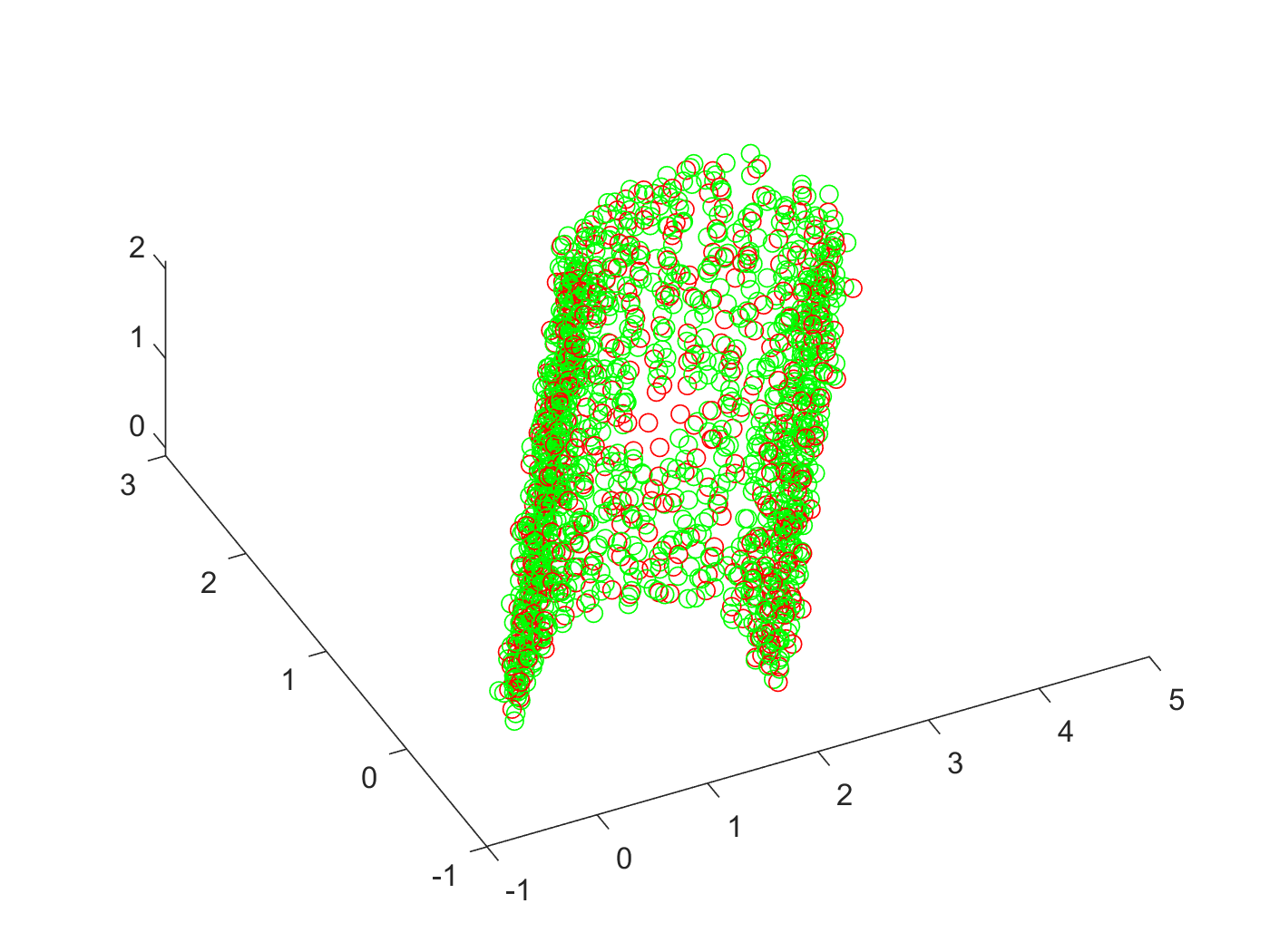}}
	\caption{Amending the scattered data from a six-dimensional cylindrical structure embedded in a 60-dimensional space. The cross-section of the cylindrical structure by a hyperplane in which the first four coordinates are greater then $-0.5$ is shown. (A) The initial scattered data of size $1180$ sampled from the manifold (the initial $P$-points in green, the initial $Q$-points in red). (B) The result after applying the plain MLOP, which produced a quasi-uniform sampling of the six-dimensional cylindrical structure: the hole is not amended. (C) The weights of the hole $T$ which are closer to $1$ near the hole location. (D) The result produced by the R-MLOP algorithm.}
	\label{fig:repair_7Dcylinder}
\end{figure}

\subsubsection*{Multiple Holes Repair}

Last, we demonstrate the ability of the R-MLOP algorithm to cope with a geometric structure of different dimensions at different locations and with multiple holes. Here we combined a 3-dimensional manifold, namely, a cone structure, with a one-dimensional manifold, namely, a line segment. This object was embedded in a 60-dimensional linear space. We used the cone parameterization
\begin{eqnarray*} p=t v_1+\frac {e^{-{R^2}}}{\sqrt 2}(\cos(u)v_2+ \sin(u)v_3)\,, \end{eqnarray*}  
where $v_1=[1,1,1,1,0, \dotsc, 0], v_2=[0, 1, -1, 0, 0,\dotsc,0], v_3=[1, 0, 0, -1, 0,\dotsc,0]$, $ (v_1,v_2,v_3) \in \R^{60}$, $t\in [0,2]$, $R\in [0,2.5]$, and $u \in [0.1\pi,1.5\pi]$. We sampled $850$ points from the structure with added uniformly distributed noise of magnitude $0.2$. The initial set $Q^{(0)}$ of size $140$ was randomly sampled (Figure \ref{fig:repair_cone_two_holes_v1}  (A)). Next, we applied the plain MLOP algorithm on our data, which produced a quasi-uniform sampling (Figure \ref{fig:repair_cone_two_holes_v1} (B))). As expected, the reconstruction maintained its proximity to the $P$-points, and did not recover the missing information. Subsequently, we calculated the coefficient $T$ using equation \eqref{eq:mu2} (its values range from $0$ to $1$ and correspond to the yellow and blue color in Figure \ref{fig:repair_cone_two_holes_v1} (C)). Finally, we executed the R-MLOP algorithm, and after 200 iterations the manifold was amended with a quasi-uniform sampling (Figure \ref{fig:repair_cone_two_holes_v1} (D)).

\vspace{-12mm}
\begin{figure}[H]
	\centering
	\captionsetup[subfloat]{farskip=0pt,captionskip=0pt, aboveskip=0pt}
	\subfloat[][]{ \includegraphics[width=0.5\textwidth]{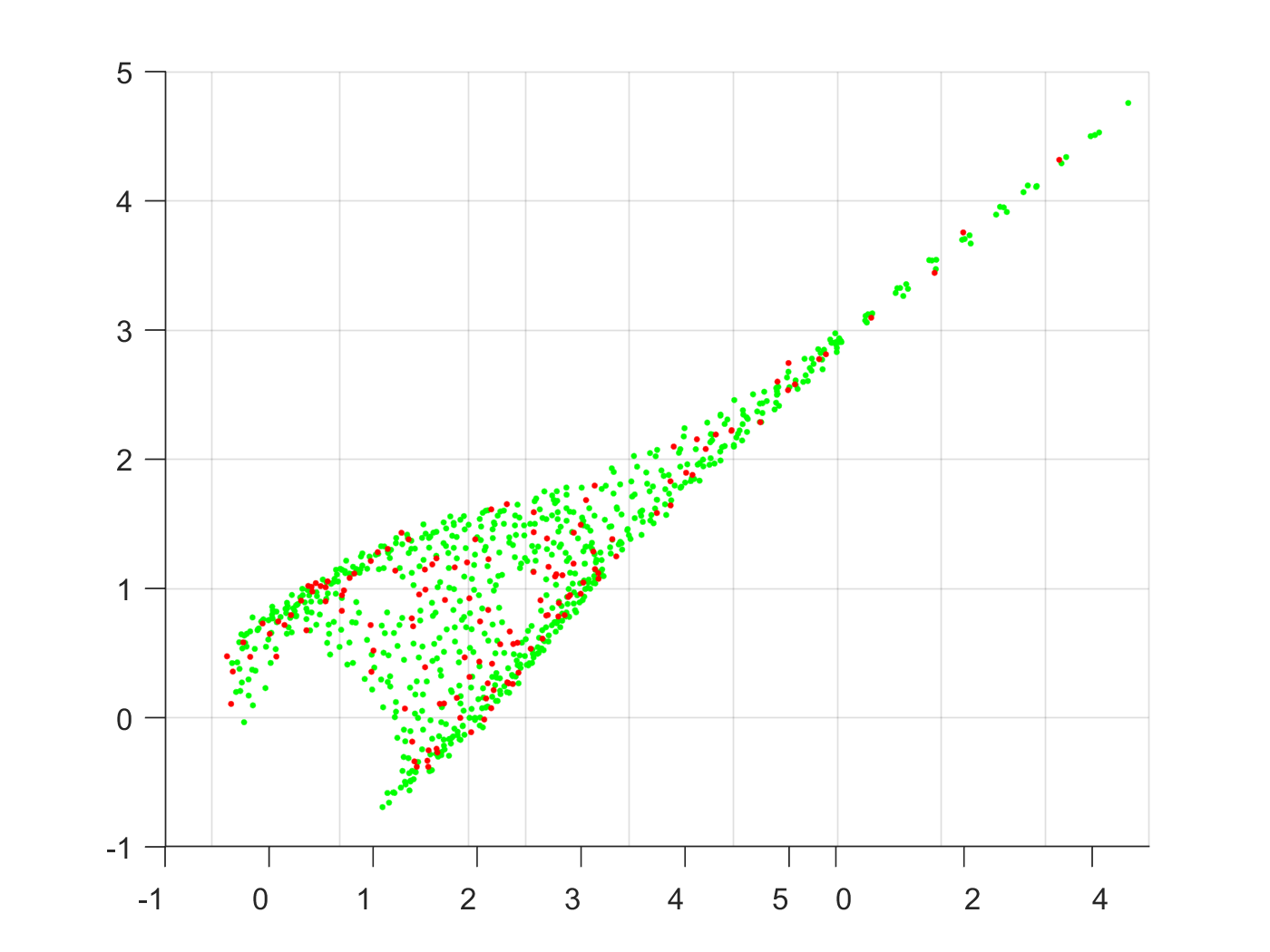} } \hspace{-2em}
	\subfloat[][]{ \includegraphics[width=0.5\textwidth]{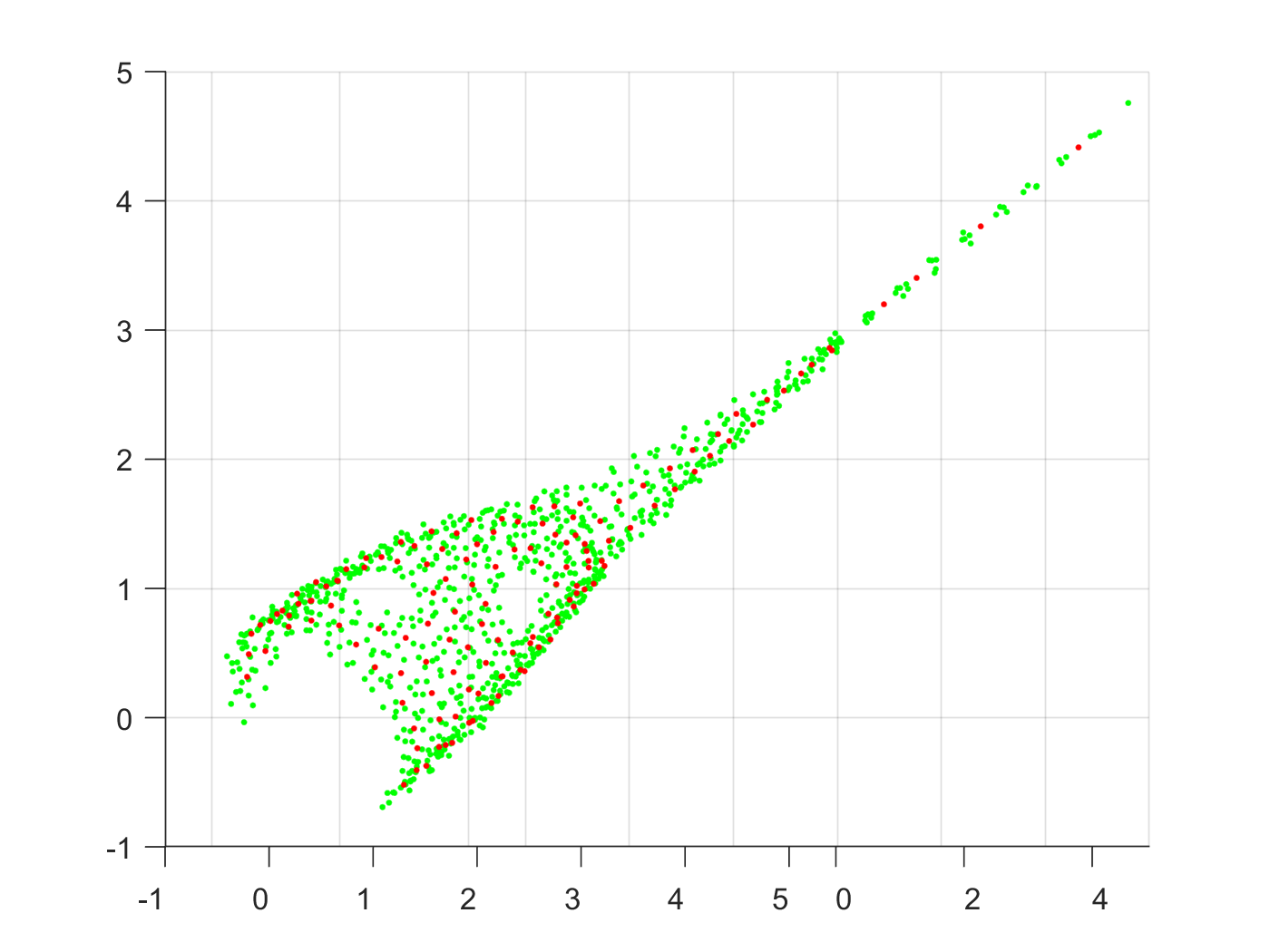} } \\[-0.7ex]
	\subfloat[][]{ \includegraphics[width=0.5\textwidth]{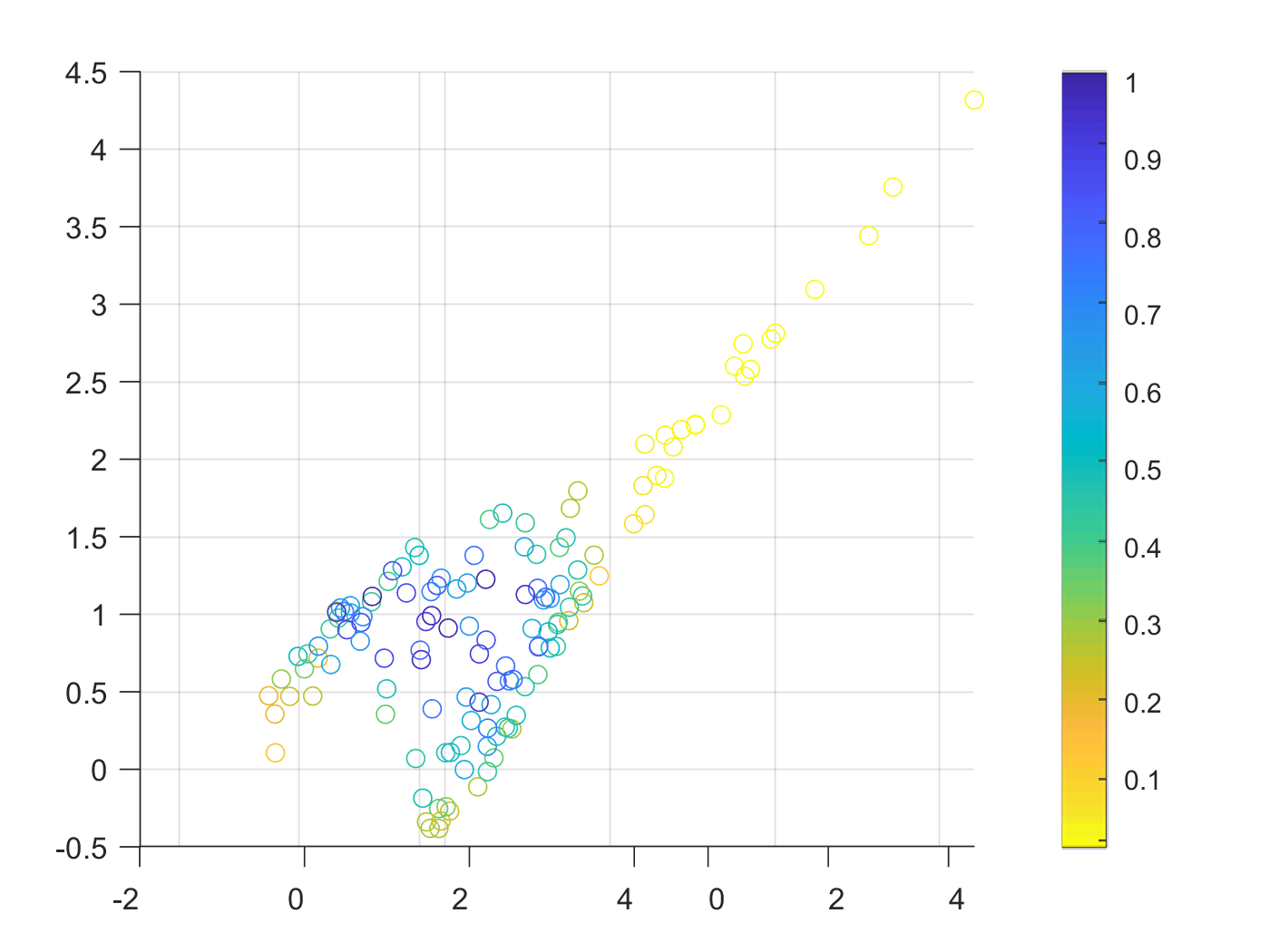} } \hspace{-2em}
	\subfloat[][]{\includegraphics[width=0.5\textwidth]{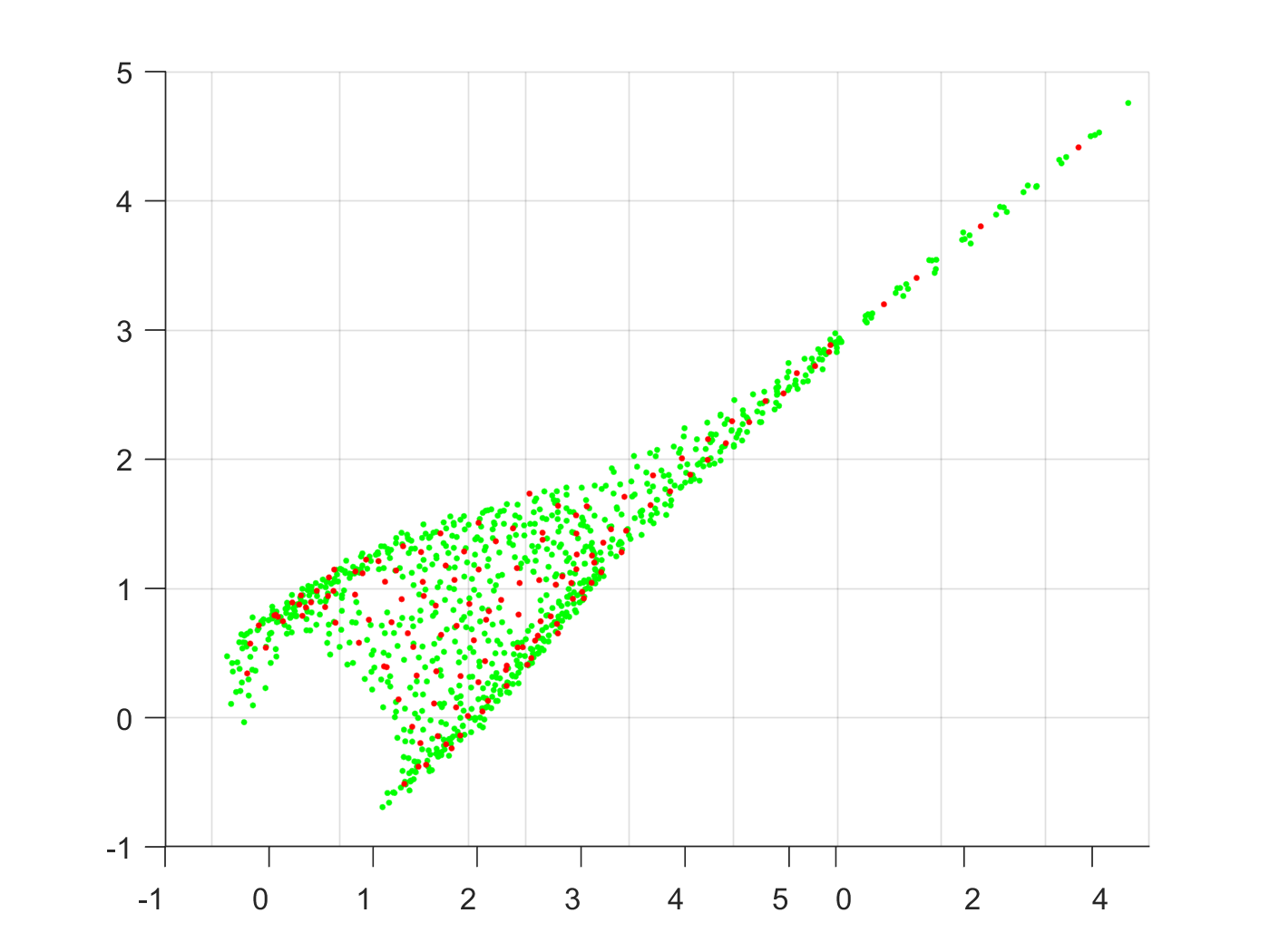}}
	\caption{Cone structure embedded into a 60-dimensional space. The first three coordinates of the point-set are shown. (A) Scattered data with uniformly distributed noise $U(-0.1; 0.1)$ (green), and the initial point-set $Q^{(0)}$ (red), with an artificial hole created. (B) The point-set generated by the MLOP algorithm. (C) The weights of the hole $T$ which are closer to $1$ near the hole location. (D) The result produced by the R-MLOP algorithm.}
	\label{fig:repair_cone_two_holes_v1}
\end{figure}

\section{Discussion and Conclusions}
\label{sec:conclusions}

This paper reviewed and developed some new results about manifold repairing in high-dimension under noisy conditions. Given a noisy point-cloud situated near a manifold of unknown intrinsic dimension, with holes, we aim at finding a noise-free reconstruction of the manifold that will amend the holes and complete the missing information. While in low-dimensional case the problem was extensively studied (in noise-free conditions), in high dimension the problem is still open. We propose a solution that extends the Manifold Locally Optimal Projection (MLOP) method \cite{faigenbaumgolovin2020manifold} by introducing an additional term that fills-in the missing data. The vanilla MLOP does not repair the manifold, since by design it maintains the proximity to the original points. To overcome this, we suggest enhancing the MLOP method to address data repairing problems by adding another force of attraction which will pull the boundary points towards their convex hull. The basic idea is to propagate information from the boundary of the hole into the hole itself, and complete missing information. In the paper, we show that the order of approximation is $O(C_1  h^2 +C_2 r^2)$, where $h$ is the representative distance between the points and $r$ is the radius of the ball that bounds the hole, and prove the validity of the proposed methodology. We demonstrate the efficiency of our method on 3D surfaces as well as on manifolds in higher dimensions, for single and multiple hole repairing.

A possible future direction would be to investigate the manifold repairing problem in the case where we observe a few sample signals. For example, suppose we observe a scatter data $P = \{p_i\}_{j=1}^J$ which is incomplete in some entries. Then is it possible to accurately guess the entries that we have not seen? This question is tightly related to the matrix competition field \cite{candes2010matrix}.

\section*{Acknowledgments}
We would like to thank Dr. Barak Sober for valuable discussions  and comments. This study was supported by a generous donation from Mr. Jacques Chahine, made through the French Friends of Tel Aviv University, and was partially supported by ISF grant 2062/18.

\bibliographystyle{spmpsci}      
\bibliography{references_MLOP}

%
%

\end{document}